\documentclass[a4paper]{amsart}

\usepackage{amsmath,amssymb}
\usepackage{enumitem,pinlabel}
\usepackage
{graphicx}
\usepackage[mathscr]{eucal}

\graphicspath{./Fig/}

\theoremstyle{plain}
  \newtheorem{thm}{Theorem}[section]
  \newtheorem{lem}[thm]{Lemma}
  \newtheorem{cor}[thm]{Corollary}
  \newtheorem{prop}[thm]{Proposition}

  \newtheorem*{obs*}{Observation}

\theoremstyle{definition}
  \newtheorem{defn}[thm]{Definition}
  \newtheorem{ex}[thm]{Example}

\theoremstyle{remark}
  \newtheorem{rem}[thm]{Remark}

\newcommand{\Z}{\mathbb{Z}}

\newcommand{\R}{\mathbb{R}}
\newcommand{\N}{\mathbb{N}}

\newcommand{\maxdeg}{\operatorname{maxdeg}}
\newcommand{\conv}{\operatorname{Conv}}
\newcommand*{\spin}{\ensuremath{\mathrm{Spin}^c}}

\allowdisplaybreaks

\begin{document}

\title{Root polytopes, parking functions, and the HOMFLY polynomial}
\author{Tam\'as K\'alm\'an}
\email{kalman@math.titech.ac.jp
\vspace*{-5pt}
}
\address{
Department of Mathematics,
Tokyo Institute of Technology,
Oh-okayama 2-12-1, Meguro-ku, Tokyo 152-8551, Japan}
\author{Hitoshi Murakami}
\email{starshea@tky3.3web.ne.jp
\vspace*{-5pt}
}
\address{
Graduate School of Information Sciences,
Tohoku University,
Aramaki aza Aoba 6-3-09, Aoba-ku,
Sendai 980-8579,
Japan}
\date{\today}

\begin{abstract}
We show that for a special alternating link diagram, the following three polynomials are essentially the same: a) the part of the HOMFLY polynomial that corresponds to the leading term in the Alexander polynomial; b)~the $h$-vector for a
triangulation of the root polytope of the Seifert graph and c) the enumerator of parking functions for the planar dual of the Seifert graph. These observations yield formulas for the maximal $z$-degree part of the HOMFLY polynomial of an arbitrary homogeneous link as well. Our result is part of a program aimed at reading HOMFLY coefficients out of Heegaard Floer homology.
\end{abstract}


\thanks{The first author is supported by a Japan Society for the Promotion of Science Grant-in-Aid for Young Scientists (B) 
and the second author by a Grant-in-Aid for Scientific Research (C)}

\maketitle


\section{Introduction}
In this paper we report on a new kind of combinatorial phenomenon in knot theory. Our research was motivated by a desire to understand what the $\text{HOMFLY}$ polynomial of an oriented link `measures,' i.e., to find a natural (for example, diagram-independent) definition for it. While that problem remains wide open (with the 
most promising approach
being the Gopakumar--Ooguri--Vafa conjecture~\cite{gv,ov}), 
we feel we did carry out an interesting case study with some surprising results.

The HOMFLY polynomial $P(v,z)$ \cite{jones} is an invariant of oriented links that specializes to the Conway polynomial $\nabla(z)$ via the substitution $\nabla(z)=P(1,z)$. The latter is equivalent to the Alexander polynomial~$\Delta(t)$ through $\Delta(t)=\nabla(t^{1/2}-t^{-1/2})$. Note that $\nabla$ and $\Delta$ share the same leading coefficient. In $P$, on the other hand, one finds several terms that contribute to the leading monomial of $\nabla$ when we set $v=1$. We will collectively refer to these as the \emph{top} of the HOMFLY polynomial. For homogeneous links \cite{cromwell} (which include all alternating and positive links), the top can also be described as the sum of those terms that realize the $z$-degree of $P$. 

A third, perhaps most adequate definition (equivalent to the previous two when the diagram is homogeneous) is that the top of $P$ is the 
coefficient of 
$z^{n-s+1}$ in $P$, where $n$ is the number of crossings and $s$ is the number of Seifert circles, respectively, in a diagram of the link. Here \emph{Seifert circles} are the simple closed curves that result when we smooth every crossing of a link diagram in the orientation-preserving way. Seifert circles are the vertices of the \emph{Seifert graph}, in which there is an edge between two of them for every crossing where they meet. A Seifert graph is always bipartite, which is the reason why it can be used in the standard construction of an \emph{oriented} spanning surface for a link. 

Let $D$ be a homogeneous link diagram \cite{cromwell}. The aim of this paper is to describe the top of the associated HOMFLY polynomial $P_D$ in terms of the Seifert graph of $D$. By definition, $D$ decomposes 
as a so-called star product of special alternating link diagrams. 
On the level of Seifert graphs, a star product corresponds to a block sum. Here a \emph{block sum} of two connected graphs (blocks) is a one-point union. Each special alternating component of $D$ is represented by a block in the Seifert graph. Every block has a sign, i.e., edges in the same block stand for crossings of the same sign (hence the term `homogeneous'). 
A theorem of Murasugi and Przytycki \cite{mp} says that the top of the HOMFLY polynomial (in the sense of our third description above) behaves multiplicatively under star product. Thus it suffices for us to describe the top of $P$ for 
special alternating links. Since such links are either positive or negative, and for any link $L$ and its mirror image~$L^*$ we have $P_{L^*}(v,z)=P_L(-v^{-1},z)$, we may, without loss of generality, concentrate on positive special alternating links only.

\begin{figure}[htbp] 
\labellist
\small
\pinlabel $r_0$ at 490 280
\pinlabel $\kappa$ at 100 110
\pinlabel $\kappa^*$ at 630 190
\endlabellist
   \centering
   \includegraphics[width=\linewidth]{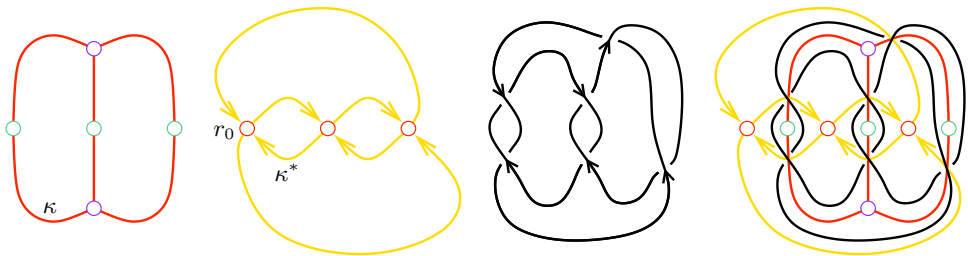} 
   \caption{A plane bipartite graph $G$, its planar dual $G^*$, and associated special alternating link $L_G$. The last panel shows the three objects together.}
   \label{fig:K32}
\end{figure}

A positive special alternating link diagram is, up to isotopy, uniquely described by its Seifert graph (as a plane graph\footnote{A \emph{plane graph} is an isotopy class of embeddings of a graph into the plane. A \emph{planar graph} is an abstract graph that admits such an embedding.}). Hence most of our discussion focuses on a connected plane bipartite graph $G$ which gives rise (by the median construction) to the positive special alternating link $L_G$. See Figure \ref{fig:K32} for an example.
Alexander Postnikov \cite{post} has recently developed a beautiful theory of (not necessarily planar) bipartite graphs. The gist of this paper is the realization that some of his ideas are closely related to knot theory.

Postnikov associates a \emph{root polytope} $Q_G$ to any bipartite graph $G$, constructed as follows. Denote the color classes of $G$ by $E$ and $V$ and take the convex hull, in $\R^E\oplus\R^V$, of the vec\-tors~$\mathbf e+\mathbf v$ for all edges $ev$ of $G$. Here $\mathbf e\in\R^E$ and $\mathbf v\in\R^V$ are the standard generators associated to $e\in E$ and $v\in V$, respectively. The result is an $(|E|+|V|-2)$-dimensional polytope so that, of course, edges of $G$ translate to vertices of $Q_G$. It is also not hard to show that a set of vertices of $Q_G$ is affinely independent if and only if the corresponding edges form a cycle-free subgraph of $G$. In particular, there is a one-to-one correspondence between maximal simplices formed by vertices of $Q_G$ and spanning trees of $G$. 

A \emph{triangulation} of $Q_G$ is a collection of maximal simplices so that their union is the entire root polytope and any two of them intersect in a common face. (Note how it is not allowed to introduce new vertices when we triangulate a polytope.) 

Our first result provides a way to triangulate $Q_G$ in the case when $G$ is a plane graph. Let us orient the dual graph $G^*$ so that each of its edges has an element of $E$ to the right and an element of $V$ to the left. After fixing a root~$r_0$ (a vertex of $G^*$) arbitrarily, we consider \emph{spanning arborescences} rooted at $r_0$. These are those spanning trees of $G^*$ in which each edge points away from $r_0$. Each spanning arborescence has a dual spanning tree in $G$ and we claim the following.

\begin{thm}\label{thm:triang}
Let $G$ be a connected plane bipartite graph. Fix a root $r_0$ and consider all spanning arborescences of $G^*$ rooted at $r_0$, as well as the spanning trees of $G$ dual to them. Then, the collection of those simplices in the root polytope $Q_G$ that correspond to the latter forms a triangulation of $Q_G$.
\end{thm}

A triangulation of a polytope is an instance of a pure simplicial complex, i.e., one in which all maximal simplices have the same dimension. To any $d$-dimensional simplicial complex, it is customary to associate the \emph{$f$-vector}
\footnote{The $f$-vector, as well as the $h$-vector below, may be more appropriately called a polynomial. However the vector terminology is so common in combinatorics that we decided to keep it.}
\[f(y)=y^{d+1}+f_0y^{d}+f_1y^{d-1}+\cdots+f_{d-2}y^2+f_{d-1}y+f_d,\]
where 
 $f_k$, for $k\ge0$, is the number of $k$-dimensional simplices in the complex. The \emph{$h$-vector} of the same complex is defined as $h(x)=f(x-1)$. The latter notion becomes significant (for example, it has positive coefficients)  for so-called \emph{shellable complexes}, that is complexes with a shelling order. Here a \emph{shelling order} of a pure simplicial complex, $\sigma_1<\sigma_2<\cdots<\sigma_{f_d}$, lists the maximal simplices in such a way that each $\sigma_i$, $i\ge1$, intersects the set~$\sigma_1\cup\cdots\cup\sigma_{i-1}$ in a union of $c_i$ codimension one faces. We always have $c_1=0$ but assume as part of the definition that $c_i\ge1$ for $i\ge2$. Whether such an order exists is a subtle question, but when it does, it is not hard to show \cite{swartz} that 
 \begin{equation}\label{eq:ftoh}
 h(x)=f(x-1)=\sum_{i=1}^{f_d}x^{d+1-c_i}.
 \end{equation}

The HOMFLY polynomial, like most knot polynomials, is usually computed via successive applications of a skein relation. The process is captured by a so-called \emph{computation tree}. The nodes of the computation tree are link diagrams with the original diagram playing the role of root. Edges in the tree correspond to simple local modifications of the diagrams. For us, the relevant skein relation is 
\begin{equation}\label{skeinstuff}
P_{\includegraphics[totalheight=8pt]{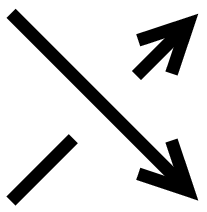}}=v^2P_{\includegraphics[totalheight=8pt]{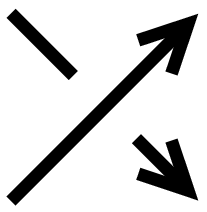}}+vzP_{\includegraphics[totalheight=8pt]{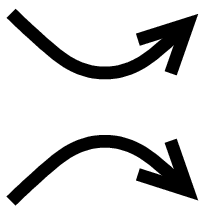}},
\end{equation}
coupled with the initial condition $P_{\includegraphics[totalheight=8pt]{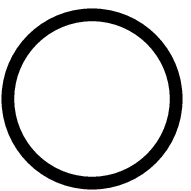}}=1$ for the HOMFLY polynomial of the unknot. Thus the computation tree is a rooted binary tree in which every non-leaf node has two descendants, resulting from either \emph{changing} or \emph{smoothing} a crossing. A priori, the crossings that we operate on can be chosen quite freely. The only restriction is that the leaves of the tree should be diagrams of the unknot, or of other links whose HOMFLY polynomials are known.

Now, the main idea of the paper is to use the spanning arborescences above to construct a computation tree $\mathscr T$ for $P_{L_G}$. Smoothing a crossing is equivalent to removing the corresponding edge from the Seifert graph. We can also keep track of crossing changes by, say, making the corresponding edge dotted. Thus each vertex of the tree will be described by a subgraph of $G$ with some dotted edges. To build $\mathscr T$, first we will use a backtrack algorithm to enumerate all arborescences of $G^*$. The subgraphs giving rise to the vertices of $\mathscr T$ will be their duals and the tree structure (as well as the dotted edges) will reflect the steps in the algorithm. 

The leaves of the tree $\mathscr T$ will arise from two kinds of subgraphs: either a spanning tree of $G$, which corresponds to an unknot diagram, or a subgraph so that along its `outside contour,' every other edge is dotted. We will not compute the HOMFLY polynomials associated to the latter (so as far as the `full' HOMFLY polynomial is concerned, $\mathscr T$ remains incomplete), but we will prove that these latter leaves do not contribute to the top of $P_{L_G}$. By contrast, the spanning trees only contribute to the top. In fact, each gives a single monomial in which the exponent of $v$ is determined by the number of dotted edges.

The tree $\mathscr T$ has a natural embedding in the plane by always drawing the result of smoothing to the right and the result of crossing change to the left. In particular, the leaves of $\mathscr T$ have a natural order from right to left. As to the leaves that belong to spanning trees, 
\begin{enumerate}
\item by Theorem \ref{thm:triang}, they correspond to the maximal simplices in a triangulation of $Q_G$ and
\item we claim that the right-to-left order 
is a shelling order for the triangulation.
\end{enumerate}
From this, the following is immediate.

\begin{prop}
The triangulation described in Theorem \ref{thm:triang} is shellable.
\end{prop}

Furthermore, as we build $Q_G$ simplex-by-simplex using the shelling order, the number $c_i$ of facets along which the simplex $\sigma_i$ is attached is exactly the number of dotted edges in the corresponding tree. From this our first main result follows:

\begin{thm}\label{thm:homfly=h}
For any connected plane bipartite graph $G$ with $s$ vertices and $n$ edges, the top of the HOMFLY polynomial $P_{L_G}(v,z)$ (of the positive special alternating link~$L_G$) is 
\[
v^{n+s-1}h(v^{-2}),\] 
where $h$ is the $h$-vector of the triangulation, given in Theorem \ref{thm:triang}, of the root polytope~$Q_G$.
\end{thm}

In general, different triangulations of the same polytope can have different $h$-vectors. For instance, the number of maximal simplices, i.e., the sum of the coefficients in the $h$-vector, may vary. That can not occur for a root polytope because maximal simplices share the same volume. In fact, much more is true: any two triangulations of $Q_G$ have the same $f$-vector and hence the same $h$-vector. We learned this fact from A. Postnikov who gave a short proof using total unimodularity and the Ehrhart polynomial. We hope to come back to this point in a future joint paper.

Our other main theorem gives a third description of the two quantities that are equated in Theorem~\ref{thm:homfly=h}. It is given in terms of \emph{parking functions} associated to $G^*$,
as defined by Postnikov and Shapiro \cite{ps}. Having fixed a root $r_0$ in $G^*$, parking functions are of the form~$R\setminus\{r_0\}\to\N=\{\,0,1,2,\ldots\,\}$, where $R$ is the vertex set of $G^*$. See Definition \ref{def:parking} for the details. Let us associate the \emph{index}
\begin{equation}\label{eq:index}
i(\pi)=\sum_{r\in R\setminus\{r_0\}}\pi(r)
\end{equation}
to the parking function~$\pi$. If $\Pi=\Pi(G^*,r_0)$ denotes the set of parking functions, then let the \emph{parking function enumerator} be 
\begin{equation}\label{eq:park}
p(u)=\sum_{\pi\in\Pi}u^{i(\pi)}.
\end{equation}

\begin{thm}\label{thm:homfly=park}
For any connected plane bipartite graph $G$ with $s$ vertices and $n$ edges, the top of the HOMFLY polynomial $P_{L_G}(v,z)$ is equal to 
\[v^{n-s+1}p(v^2),\] 
where $p$ is the parking function enumerator of the directed graph $G^*$.
\end{thm}

A consequence of the Theorem is the 
fact 
that $p(u)$ is independent of the choice of $r_0$. 
As a biproduct of our arguments, we obtain a bijection between spanning arborescences of $G^*$ rooted at $r_0$ and parking functions defined on $R\setminus\{r_0\}$. It has been known \cite{ps} that those two sets have the same cardinality. Our bijection is similar to, but appears not to be a special case of, those in the literature \cite{jonevevan}. We also obtain


\begin{cor}\label{cor:h=p}
Let $G$ be a connected plane bipartite graph on $s$ vertices with planar dual $G^*$. Then the $h$-vector $h(x)$ of any triangulation of the root polytope~$Q_G$ and the parking function enumerator $p(u)$ of $G^*$ satisfy \[u^{s-1}h(u^{-1})=p(u).\]
\end{cor}

The combinatorial setup used in this paper yields yet another description of the top of $P_{L_G}$, this time as the \emph{interior polynomial} \cite{hypertutte} of the hypergraph $(V,E)$. \label{int_conj}
This last claim is only conjecturally true, but once it is proved, it will provide a hitherto unknown connection between the HOMFLY polynomial and Heegaard Floer homology. Namely, the interior polynomial can be computed from the so-called hypertree polytope (see \cite{hypertutte} for definitions) and the latter can be thought of as the Heegaard Floer homology of a certain sutured manifold\footnote{\label{foot} In order to be more precise, let us note that the sutured manifold in question is a handlebody, so that the set of $\spin$ structures supporting the sutured Floer homology is a subset of a lattice. This set is isomorphic to the set of lattice points in the hypertree polytope. Furthermore, the homology group corresponding to each $\spin$ structure in the support is $\Z$.} \cite{jkr}. Due to Theorem~\ref{thm:homfly=h}, the problem of reading (some) HOMFLY coefficients out of Heegaard Floer homology is reduced to the open combinatorial problem that the $h$-vector (of a triangulation of the root polytope $Q_G$ of the connected bipartite graph $G$) is equivalent to the interior polynomial (of the hypergraph $(V,E)$, where $V$ and $E$ are the color classes of $G$). 
 
The sutured manifold mentioned above is the complement of a Seifert surface. However this surface is bounded not by $L_G$ but by a related link. One may wish to consider instead the minimum genus Seifert surface $F_G$ for $L_G$ that deformation retracts to $G$. The sutured Floer homology $S_G$ of the complement of $F_G$ is also a hypertree polytope (in the sense of footnote \ref{foot}) but of the wrong hypergraph, whose interior polynomial is different from the top of $P_{L_G}$. On the other hand, the set $\Pi(G^*)$ of parking functions can be thought of as a rearrangement\footnote{We plan to clarify the meaning of this claim in a future joint paper with Dylan Thurston. Here let us only note that $S_G$ and $\Pi(G^*)$ both have dimension $|R|-1$.} of $S_G$ and thus Theorem \ref{thm:homfly=park} also becomes a way of obtaining information on the HOMFLY polynomial from Heegaard Floer homology.

We end the introduction with a statement of our main result for homogeneous links. As explained above, this follows directly from our other claims via Murasugi--Przytyczki's product formula.

\begin{thm}\label{thm:homog}
Let $D$ be a homogeneous link diagram with Seifert graph $G$ that is composed of $k$ positive and $l$ negative blocks. 
Let $p_i(v)$, $1\le i\le k$ and $p_j'(v)$, $1\le j\le l$ be the parking function enumerators of the dual graphs of each block. By Corollary \ref{cor:h=p}, these polynomials can also be interpreted as $h$-vectors. 

Now if $G$ has altogether $s_+$ vertices and $n_+$ edges in its positive blocks and $s_-$ vertices and $n_-$ edges in its negative blocks (here type II Seifert circles, i.e., vertices where blocks are attached, are counted once for each block they belong to) so that the writhe of $D$ 
is $w(D)=n_+-n_-$, then the coefficient of $z^{n-s+1}$ in the HOMFLY polynomial~$P_D(v,z)$ (which is the highest power of $z$ occurring in $P$) is \[(-1)^{n_--s_-+l}\cdot v^{w(D)-s_++s_-+k-l}\cdot\prod_{i=1}^kp(v^2)\cdot\prod_{j=1}^lp'(v^{-2}).\]
\end{thm}

The paper is organized as follows. We start with material on arborescences and organize them in a binary tree in Section \ref{sec:arbo}. In Section \ref{sec:comptree}, we use the binary tree to compute the top of the HOMFLY polynomial of a special alternating link. The proof of an important proposition will be delayed until Section \ref{hitoshi}. In Section \ref{sec:rootpoly}, we recall some of Postnikov's results and prove Theorem \ref{thm:triang}. In Sections \ref{sec:triang} and \ref{sec:park}, respectively, we use the results of Section \ref{sec:comptree} to establish Theorems \ref{thm:homfly=h} and \ref{thm:homfly=park}.

{\bf Acknowledgements.} We are grateful to Alexander Postnikov from whom we learned about the $h$-vector. The paper also benefited greatly from conversations with Dylan Thurston.

\section{Arborescences}\label{sec:arbo}

Most arguments in the present paper are centered around a binary tree. We will give three mutually isomorphic descriptions of it. In this section, the binary tree will appear as the `tree of arborescences.' In the next, we will also describe it as the `tree of subgraphs' and then as the `HOMFLY computation tree.' But first, we introduce an object which will become a distinguished leaf of the tree of arborescences.

\subsection{The clocked arborescence}\label{ssec:clock}

All graphs that appear in this paper are finite. 
Multiple edges and loop edges are allowed. (Of course, the latter do not occur in bipartite graphs.) A \emph{subgraph} of a graph will always have the same vertex set as the original, i.e., a subgraph will just be a subset of the edges of the graph. A \emph{spanning tree} is a connected and cycle-free subgraph.

By definition, the edges of any plane graph $G$ and the edges of its planar dual $G^*$ are in a one-to-one correspondence. This gives rise to a bijection between subgraphs, where a set of edges of $G$ is paired with the complementer set of the corresponding edges of $G^*$. Elements of such a pair will be called \emph{dual subgraphs}. If $G$ and $G^*$ are both connected, then it is well known that a subgraph (of $G$ or of $G^*$) is a spanning tree if and only if its dual is one.

\begin{defn}
Let $J$ be a directed graph (possibly with loop edges and multiple edges) and let us fix a vertex $r_0$, called the \emph{root}, in $J$. An \emph{arborescence} rooted at $r_0$ is a subgraph of $J$ so that
\begin{itemize}
\item its connected components not containing $r_0$ are isolated points and
\item its connected component containing $r_0$, called the \emph{root component}, is a tree in which there is a (unique) directed path from $r_0$ to any other vertex.
\end{itemize}
A \emph{spanning arborescence} is an arborescence without isolated points.
\end{defn}

We remark that arborescences may never contain loop edges.

Let now $G$ be a plane bipartite graph so that $G^*$ is directed as explained in the introduction: if $E$ and $V$ are the color classes of $G$, then as we traverse each edge of $G^*$, we see an element of $E$ to our right and an element of $V$ to our left. We may write a directed edge as an ordered pair $(\text{startpoint},\text{endpoint})$, although we should always keep it in mind that multiple edges may exist with the same initial and terminal points. Let us also recall that the vertex set of $G^*$ is identified with the set $R$ of regions of $G$ (i.e., the set of connected components of $S^2\setminus G$).

\begin{lem}\label{lem:connected}
Let $G$ be a plane bipartite graph. There exists a directed path from any vertex of $G^*$ to any other vertex.
\end{lem}

\begin{proof}
Assume the contrary, i.e., that there exists a vertex $r_0\in G^*$ so that the set~$R'\subset R$ of vertices that are accessible from $r_0$ with directed paths is not $R$. Then the union of the corresponding (to elements of $R'$) regions of $G$ is not the entire sphere $S^2$ and hence it has non-empty boundary. That boundary is a collection of cycles in $G$ and since $G$ is bipartite, each boundary component consists of at least two edges. It is easy to see that half of the edges along each boundary component are such that the corresponding edge of $G^*$ points from an element of $R'$ to an element of $R\setminus R'$. But that is a contradiction because by definition, the endpoints of these edges of $G^*$ should be in $R'$.
\end{proof}

Fix a root $r_0$ in $G^*$ and an edge $\kappa$ of $G$ so that the dual edge $\kappa^*$ points to $r_0$. We will use a``greedy,'' or depth-first, algorithm to construct a spanning arborescence of $G^*$ determined by these data. In the process, we will select edges one-by-one so that at each stage we have an arborescence of $G^*$.

Let $\varepsilon_{0,1}$ be the first non-loop edge of $G^*$, directed away from $r_0$, that we find as we turn around $r_0$, starting from $\kappa^*$, in the positive (counterclockwise) direction. Let $r_1$ be the endpoint of $\varepsilon_{0,1}$. Now turn counterclockwise around $r_1$, starting from $\varepsilon_{0,1}$, until the first edge $\varepsilon_{1,1}$ is found so that together with $\varepsilon_{0,1}$ they form an arborescence (i.e., $\varepsilon_{1,1}$ is neither a loop nor does it point to $r_0$). Now move to the endpoint $r_2$ of $\varepsilon_{1,1}$ and turn around it counterclockwise, starting from $\varepsilon_{1,1}$, until the next edge is found so that together with the first two, they form an arborescence, and so on.

If at any point in the process we select the edge $\varepsilon_{i,j}=(r_i,r_k)$ but complete a full turn around $r_k$ without finding a suitable next edge, then we move back to $r_i$ and continue turning counterclockwise around it from $\varepsilon_{i,j}$ until an edge $\varepsilon_{i,j+1}$ is found which forms an arborescence with the previously chosen ones. If this does not exist either, then we move back to the starting point $r_l$ of the unique edge~$\varepsilon_{l,m}=(r_l,r_i)$ in our arborescence which ends at $r_i$ and continue searching for a suitable edge~$\varepsilon_{l,m+1}$ by turning counterclockwise around $r_l$, starting from $\varepsilon_{l,m}$, etc.

The edge-selecting algorithm terminates when a full turn has been completed around all vertices of $G^*$ that we visited in the process (including $r_0$). We claim that the final arborescence $A$ is spanning. Indeed, if there was an isolated point~$r$ in $A$ then find a directed path in $G^*$ from $r_0$ to $r$ (cf.\ Lemma \ref{lem:connected}). Tracing this path backward from $r$, the first edge that is in $A$ is preceded by an edge that could be added to $A$ to form a larger arborescence (if $A$ and the path are disjoint, then the same can be said about the first edge along the path). Moreover, when we were turning around the startpoint of this edge, we would have selected it into our arborescence, which is a contradiction.

\begin{defn}
Let $G$ be a connected plane bipartite graph with dual graph~$G^*$. The spanning arborescence of $G^*$ constructed above will be called the \emph{clocked arborescence} (relative to the vertex $r_0$ of $G^*$ and the edge $\kappa$ of $G$). 
\end{defn}

\subsection{The tree of arborescences}\label{ssec:arbo}

We start with a technical digression. We stress that (say, smooth) embeddings of the graphs $G$ and $G^*$ into the sphere~$S^2$ have been fixed. Let $C$ denote the set of intersection points between edges of $G$ and $G^*$. With an abuse of notation, we will also speak of $C$ as the edge set of either $G$ or $G^*$. Likewise, we use the symbol $R$ both for the regions of $G$ and for the vertices of $G^*$. The regions of $G^*$ have their boundary oriented either clockwise or counterclockwise. These two sets of regions are identified with the color classes~$E$ and $V$, respectively, of $G$.

Now, let us fix smooth arcs connecting each vertex $r\in R$ to the vertices of $G$ that lie along the boundary of the region $r$ (but otherwise avoiding $G$ and $G^*$). Together with $G$ and $G^*$, these arcs form a triangulation $T_R$ of $S^2$. Indeed, the set of $0$-cells is $E\cup V\cup R\cup C$ and the $1$-cells are the arcs above along with the half-edges of $G$ and $G^*$ emanating from elements of $C$. Let us also fix a barycentric subdivision $\mathscr B$ of $T_R$.

\begin{defn}\label{def:regnbhd} 
The \emph{regular neighborhood} $N_A$ of an arborescence $A$ in $G^*$ is the union of those (closed) $2$-cells of $\mathscr B$ that have a common point with the root component of $A$. 
\end{defn}

In combinatorial topology, a regular neighborhood is usually defined using a second barycentric subdivision. However for our purposes, Definition \ref{def:regnbhd} will suffice and so we will use it to avoid unnecessary complication.

Next, we shall describe an algorithm that enumerates all arborescences of $G^*$ rooted at $r_0$ and arranges them in a binary tree $\mathscr A$. The construction will depend on the same edge $\kappa$ of $G$ as above. The nodes of the binary tree will actually be pairs $(\text{arborescence},\text{set of skipped edges})$ so that
the skipped edges are edges of $G^*$, not in the arborescence, each of which has its startpoint in the root component. The nodes of $\mathscr A$ have either no descendant or exactly two, which we will refer to as the right and left descendants.

Our first arborescence, the root of the binary tree, is the one with no edges and no skipped edges. 
The right descendant of the root has the unique edge $\varepsilon_{0,1}=(r_0,r_1)$ that appeared in the construction of the clocked arborescence. It has no skipped edges. 
The left descendant of the root still has no edges, 
but it has the skipped edge $\varepsilon_{0,1}$.
The following is the general description of our process.

\begin{defn}\label{def:descend}
If $(A,S)$ has already been constructed as a node of the tree $\mathscr A$, then let $N_A$ denote the regular neighborhood of $A$.
Let $k$ be the intersection point of the boundary $\partial N_A$ and the edge $\kappa^*$ (which can never be in $A$, hence $k$ exists). Now, let us move counterclockwise around $\partial N_A$ starting from $k$ 
until we reach the first edge $\delta$ of $G^*$ that is not in $S$ and which is such that $A\cup\{\delta\}$ is an arborescence. We will refer to $\delta$ as the \emph{augmenting edge} of $(A,S)$. If such a $\delta$ does not exist, then the node $(A,S)$ will have no descendants in $\mathscr A$. Otherwise, let the \emph{right descendant} of $(A,S)$ be $(A\cup\{\delta\},S)$ and let its \emph{left descendant} be $(A,S\cup\{\delta\})$.
\end{defn}

\begin{ex}\label{ex:arb_tree}
\begin{figure}[htbp] 
\labellist
\small
\pinlabel {type I} at 980 1019
\pinlabel {type I} at 980 227
\pinlabel {type I} at 1376 623
\pinlabel {type II} at 1376 227
\pinlabel {type II} at 584 227
\pinlabel {type II} at 188 227
\endlabellist
   \centering
   \includegraphics[width=4in]{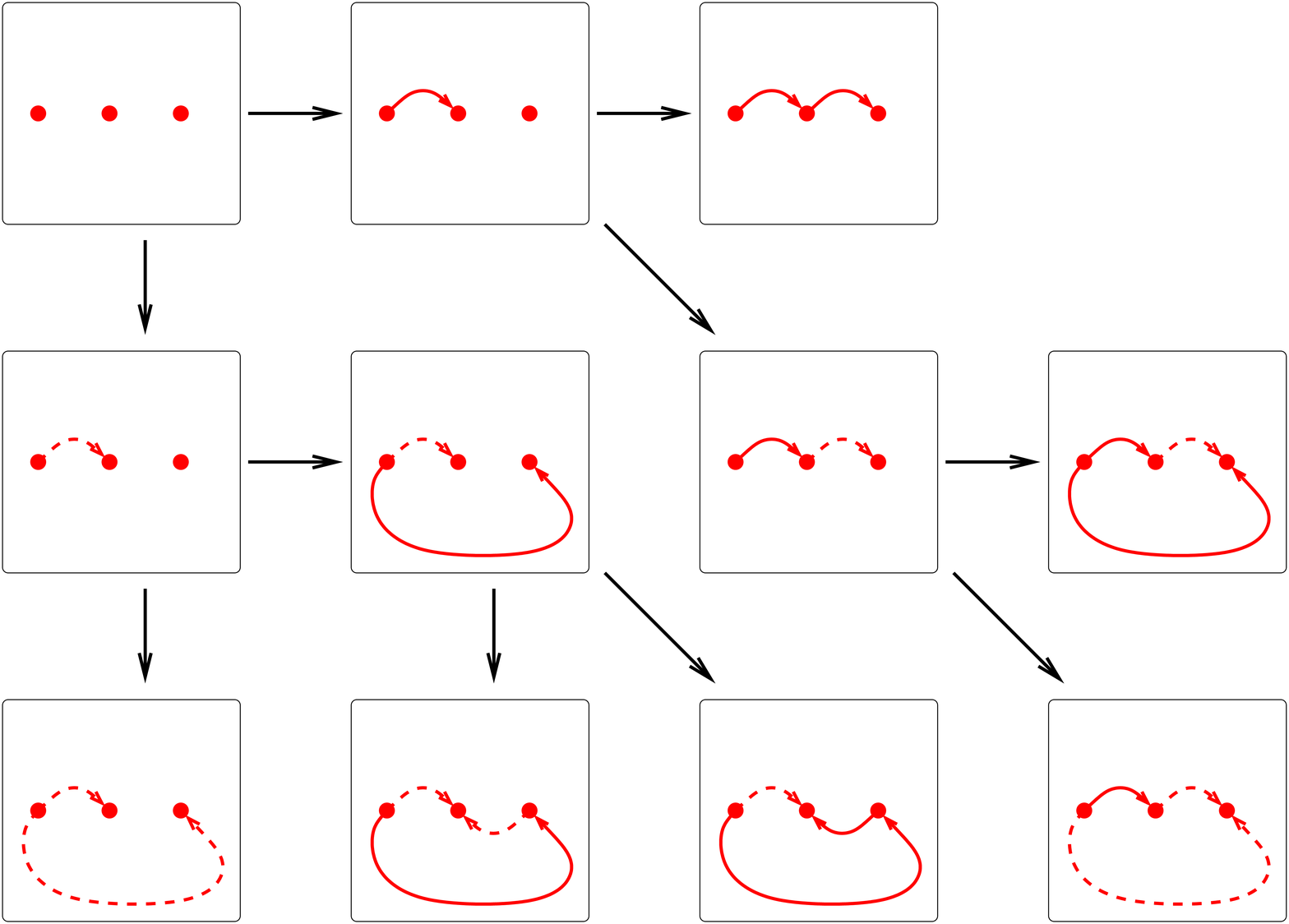} 
   \caption{The tree of arborescences $\mathscr A$ for the directed graph in Figure \ref{fig:K32}. For each node $(A,S)$, solid edges represent $A$ and dotted edges represent $S$.}
   \label{fig:arb_tree}
\end{figure}

Figure \ref{fig:arb_tree} shows the tree of arborescences for the complete bipartite graph $G=K_{3,2}$. The embedding of $G$ (and $G^*$), the root $r_0$, and the edge $\kappa$ are shown in Figure~\ref{fig:K32}.
\end{ex}

Notice that along any path in $\mathscr A$ that starts from the root, both the arborescences and the sets of skipped edges form an increasing sequence. It is easy to see that the rightmost branch of $\mathscr A$, i.e., the path when we always pass to the right descendant, leads to the clocked arborescence (with no skipped edges). As to other potential terminal points of such paths, we make the following observation.

\begin{lem}
The node $(A,S)$ of $\mathscr A$ has no descendants if and only if either
\begin{enumerate}[leftmargin=30pt,label=\Roman*.]
\item $A$ is a spanning arborescence or
\item the set $R'$ of vertices in the root component of $A$ is a proper subset of $R$ but all edges of $G^*$ from an element of $R'$ to an element of $R\setminus R'$ belong to $S$.
\end{enumerate}
\end{lem}

With regard to the above, we will speak of \emph{type I} and \emph{type II leaves} of $\mathscr A$.

\begin{proof}
A spanning arborescence is also a spanning tree of $G^*$ and hence it contains exactly $|R|-1$ edges. Other arborescences have fewer edges. Consequently, spanning arborescences can not be extended as arborescences and so cannot have descendants in $\mathscr A$. If $(A,S)$ fits the second description then it has no augmenting edge and hence no descendants, either.

As to the converse, let now $(A,S)$ be so that $A$ is not spanning
and let $\alpha$ be an edge from the root component to an isolated point of $A$. Such an edge 
necessarily intersects the boundary $\partial N_A$ of the regular neighborhood of $A$ and hence it will be detected by the process described in Definition \ref{def:descend}. Thus if $(A,S)$ has no descendants in $\mathscr A$ then any possible $\alpha$ belongs to $S$.
\end{proof}

\begin{lem}\label{lem:allarboappear}
All spanning arborescences of $G^*$ appear at a unique node (a type I leaf) of $\mathscr A$.
\end{lem}

\begin{proof}
Let $A$ be a spanning arborescence of $G$. We will construct a path in $\mathscr A$ that starts from the root and ends at $(A,S)$ for an appropriate $S$. Assume that the path has already been constructed until the node $(A',S')$ and let $\delta$ be the augmenting edge of $(A',S')$. If $\delta$ is an edge in $A$, then continue the path to the right descendant~$(A'\cup\{\delta\},S')$. Otherwise, move to the left descendant $(A',S'\cup\{\delta\})$. 

We claim that the last node $(\widetilde A,\widetilde S)$ along our path cannot be a leaf of type~II. Since $\widetilde A\subset A$ by construction, if $\widetilde A\ne A$, then there has to be an edge of $A$ going from a vertex of the root component of $\widetilde A$ to an isolated point of $\widetilde A$. As our process never skips edges of $A$, this edge cannot be in $\widetilde S$ either, but that contradicts the definition of a type II leaf. Hence $(\widetilde A,\widetilde S)$ is a type I leaf which implies $\widetilde A=A$.

Regarding uniqueness, assume that there exists another path $P$ in $\mathscr A$ from the root to a leaf involving $A$. Let $(A',S')$ be the node where the two paths ($P$ and the one constructed above) separate. Since $(A',S')$ cannot be a leaf, it has an augmenting edge $\delta$. Now the next node $(A'',S'')$ along $P$ after $(A',S')$ is such that either 
\begin{itemize}
\item $\delta\not\in A$ but $\delta\in A''$ (if $P$ takes a step to the right at $(A',S')$) or
\item $\delta\in A$ but $\delta\in S''$ (if $P$ takes a step to the left).
\end{itemize}
As both scenarios prevent the endpoint of $P$ from involving $A$, we have a contradiction and the proof is complete.
\end{proof}


If the arborescence $A$ is not spanning then it may appear at multiple nodes of $\mathscr A$, as the example of the root and its left descendant already shows. 

%

\section{The computation tree}\label{sec:comptree}

In this section we describe two more incarnations of the binary tree $\mathscr A$ and we spell out their relation to the top of the HOMFLY polynomial $P_{L_G}$.

It is a straightforward matter to transform the tree $\mathscr A$ of arborescences in $G^*$ into the isomorphic tree $\mathscr G$ of (decorated) subgraphs of $G$. If $(A,S)$ is a node in $\mathscr A$, then replace it with the dual subgraph $A^*$ of $G$. The edges of $G$ that are dual to elements of $S$ are in $A^*$ and we will refer to them as the \emph{dotted edges}.

The tree structures of $\mathscr A$ and $\mathscr G$ are the same, that is, Definition \ref{def:descend} serves to describe $\mathscr G$ as well. However it is useful to translate that description to subgraph terms. Recall that $G$ is embedded in $S^2$. Let us refer to the region of $G$ marked with $r_0$ as the `initial outside region.' At each stage of the process, we look for an edge of $G$ to `puncture' (i.e., remove) so that the outside region grows larger. To find such an edge, we travel around the boundary of the outside region in the counterclockwise direction, starting from $\kappa$, and we select the first non-dotted edge~$\delta^*$ which is aligned so that its endpoint in $E$ is encountered first. It is also required that $\delta^*$ be adjacent to the outside region on one side only, so that after puncturing it, the subgraph remains connected and the outside region remains simply connected. In the right descendant of the subgraph, $\delta^*$ is removed. In the left descendant, the subgraph is the same but $\delta^*$ becomes dotted, so that it can not be punctured anymore along the current branch of $\mathscr G$. 

A (decorated) subgraph is a leaf of $\mathscr G$ if no suitable $\delta^*$ can be found. That can happen in two ways. Either the only region of the subgraph is the outside region, i.e., the subgraph is a tree -- these are the type I leaves. Or else, the closure of the outside region is not the entire sphere $S^2$ but along each boundary component of the closure, dotted and non-dotted edges alternate. In this case, that is in the case of a type II leaf, we say that the subgraph has an \emph{alternating contour}.

Finally, in order to turn $\mathscr G$ into the computation tree $\mathscr T$, we replace each subgraph with an oriented link diagram using the median construction. That is, given a node of $\mathscr G$, we consider a regular neighborhood of the corresponding subgraph and we give a half-twist to that surface over each edge. The half-twist is positive for non-dotted edges and negative for dotted ones. The link diagram is the oriented boundary of the surface so constructed. In particular, the link diagram at the root of $\mathscr T$ is $L_G$. In terms of these diagrams, passing to a right descendant means \emph{smoothing} a crossing and passing to a left descendant is equivalent to \emph{changing} a crossing (from positive to negative).

Smoothing and changing of crossings play a crucial role in the definition of the HOMFLY polynomial $P(v,z)$, as we explained on page \pageref{skeinstuff} of the Introduction. Let us quote the following well known result.

\begin{thm}[Morton \cite{morton}]\label{thm:morton}
If an oriented link diagram contains $n$ crossings and $s$ Seifert circles, then in any term of the corresponding HOMFLY polynomial, the exponent of $z$ is at most $n-s+1$.
\end{thm}

Recall that homogeneous links \cite{cromwell} are built from special alternating ones using the operation called `star product.'
Morton's estimate above is well known to be sharp for homogeneous link diagrams. This claim also follows from our main result on special alternating links (as stated in Theorems \ref{thm:homfly=h} and \ref{thm:homfly=park}, although the statement in Theorem \ref{thm:treetop} suffices as well) and the following fact.

\begin{thm}[Murasugi--Przytyczki \cite{mp}]\label{thm:mp}
Let $D_1$ and $D_2$ be oriented link diagrams so that they have $n_1$ and $n_2$ crossings, respectively, as well as $s_1$ and $s_2$ Seifert circles, respectively. Let us form a star product $D_1\star D_2$ of $D_1$ and $D_2$ (along any pair of type I Seifert circles and using any arrangement of the crossings along the new type II Seifert circle). Then $D_1\star D_2$ has $n_1+n_2$ crossings and $s_1+s_2-1$ Seifert circles, and the coefficient of $z^{n_1+n_2-(s_1+s_2-1)+1}$ in $P_{D_1\star D_2}(v,z)$ is the product of the coefficients of $z^{n_1-s_1+1}$ in $P_{D_1}(v,z)$ and that of $z^{n_2-s_2+1}$ in $P_{D_2}(v,z)$.
\end{thm}

\begin{defn}
For a link $L$ that can be presented with a homogeneous diagram with $n$ crossings and $s$ Seifert circles, the \emph{top of the HOMFLY polynomial} is the polynomial in $v$ that is the coefficient of $z^{n-s+1}$ in $P_{L}(v,z)$.
\end{defn}

When we combine the next theorem with Theorem \ref{thm:mp}, it follows that for any homogeneous link, all coefficients in the top of the HOMFLY polynomial have the same sign. In particular, they do not cancel when we pass to the Alexander polynomial $\Delta(t)=P(1,t^{1/2}-t^{-1/2})$, rather their sum becomes the leading coefficient in $\Delta$. 


\begin{thm}\label{thm:treetop}
For a connected plane bipartite graph $G$ of $s$ vertices and $n$ edges, the top of the HOMFLY polynomial~$P_{L_G}$ is expressed in terms of the computation tree $\mathscr T$ as the sum to which each type~I leaf with $k$ negative crossings (which appear as skipped edges in $\mathscr A$ and as dotted edges in $\mathscr G$), contributes the monomial $v^{n-s+1+2k}$.
\end{thm}

\begin{proof}
If we label each edge of $\mathscr T$ connecting a node to its right descendant with $vz$ and each edge leading to a left descendant with $v^2$, then $\mathscr T$ can be used to compute the HOMFLY polynomial associated to the root as the sum of the following terms: For each leaf, take the HOMFLY polynomial of the corresponding link and multiply it with the product of the edge labels along the unique path between the leaf and the root. 

Type I leaves, where the corresponding subgraph is a spanning tree of $G$, are diagrams of the unknot. Therefore, since $P_{\includegraphics[totalheight=8pt]{circle.eps}}=1$, type I leaves contribute a single monomial, namely the product of the appropriate edge labels. Each spanning tree contains $s-1$ edges. So in order to reach a type I leaf, one needs to remove $n-s+1$ edges, that is, on the way from the root to the leaf, one has to take a step to the right exactly $n-s+1$ times. If the link diagram at the leaf contains $k$ negative crossings, that means that we took a step to the left $k$ times. Hence the contribution to the HOMFLY polynomial of the leaf is $(vz)^{n-s+1}(v^2)^k=v^{n-s+2k+1}z^{n-s+1}$. 
Note that the exponent of $z$ here is the maximum allowed by Theorem \ref{thm:morton}.

The Theorem will obviously follow once we make sure that the type II leaves of the computation tree do not contribute to the top of $P_{L_G}$. This is a consequence of Proposition \ref{biglemma} below by the following remarks.

Along any path from the root of $\mathscr T$ to a node in the tree, the exponent of $z$ in the product of the corresponding edge labels is the number of steps taken to the right. That number agrees with the amount $n-n'$ by which the number of crossings decreased along the path. The number $s$ of Seifert circles is constant throughout $\mathscr T$. Hence in order for the node to contribute terms containing $z^{n-s+1}$ to $P_{L_G}$, the HOMFLY polynomial associated to the node has to contain terms with $z^{n'-s+1}$ in them. The estimate in Proposition \ref{biglemma} rules that out for type II leaves.

The Proposition below describes some characteristics of the link diagrams that arise after (in the description of $\mathscr G$ at the beginning of this section) an alternating contour has been achieved. There is however an extra assumption in the Proposition, namely that the outside region at this stage does not touch itself over an edge of $G$. General type~II leaves of $\mathscr T$ are obtained by connecting diagrams described in the Proposition in a tree-like fashion. Here by connecting, we mean joining the corresponding (embedded) Seifert graphs by paths of edges (some of which may be dotted). In terms of link diagrams, that translates to joining by a sequence of ($0$ or more) bigons, which is just a complicated way of taking a connected sum. Because the HOMFLY polynomial is multiplicative under connected sums, the estimate discussed in the previous paragraph follows for general type II leaves from the claim below by a short inductive argument (it even gets stronger as the number of components in the alternating contour increases).
\end{proof}

\begin{prop}\label{biglemma}
Let the oriented link diagram $D$ contain $n$ crossings and $s$ Seifert circles. Assume that there exists a region $r$ in the complement $S^2\setminus D$ with the following property: There are at least two arcs of $D$ bounding $r$ and they are alternately oriented clockwise and counterclockwise, so that at each crossing along $\partial r$, the counterclockwise arc passes over the clockwise one. We also assume that near each crossing along $\partial r$, only one quadrant formed by $D$ belongs to $r$. (See Figure~\ref{fig:outer_arcs}.) Then, in any term of the corresponding HOMFLY polynomial $P_D(v,z)$, the exponent of $z$ is at most $n-s-1$.
\end{prop}

This statement is just a slight improvement on Morton's upper bound in Theorem~\ref{thm:morton}, yet it turns out to be where most of the difficulty in this paper is concentrated. Notice that we did not make any assumption on the crossings not adjacent to $r$, even though when we apply Proposition \ref{biglemma} in the proof of Theorem \ref{thm:treetop}, away from $\partial r$ the diagram $D$ is still alternating. Likewise, it is not assumed that $D$ be special. We delay the proof until Section \ref{hitoshi}.

\begin{ex}\label{ex:K32}
\begin{figure}[htbp] 
\labellist
\small
\pinlabel $v^2$ at 81 350
\pinlabel $v^2$ at 81 746
\pinlabel $v^2$ at 477 350
\pinlabel $v^2$ at 680 726
\pinlabel $v^2$ at 1076 330
\pinlabel $vz$ at 300 570
\pinlabel $vz$ at 300 966 
\pinlabel $vz$ at 1076 570
\pinlabel $vz$ at 690 966 
\pinlabel $vz$ at 680 330
\endlabellist
   \centering
   \includegraphics[width=\linewidth]{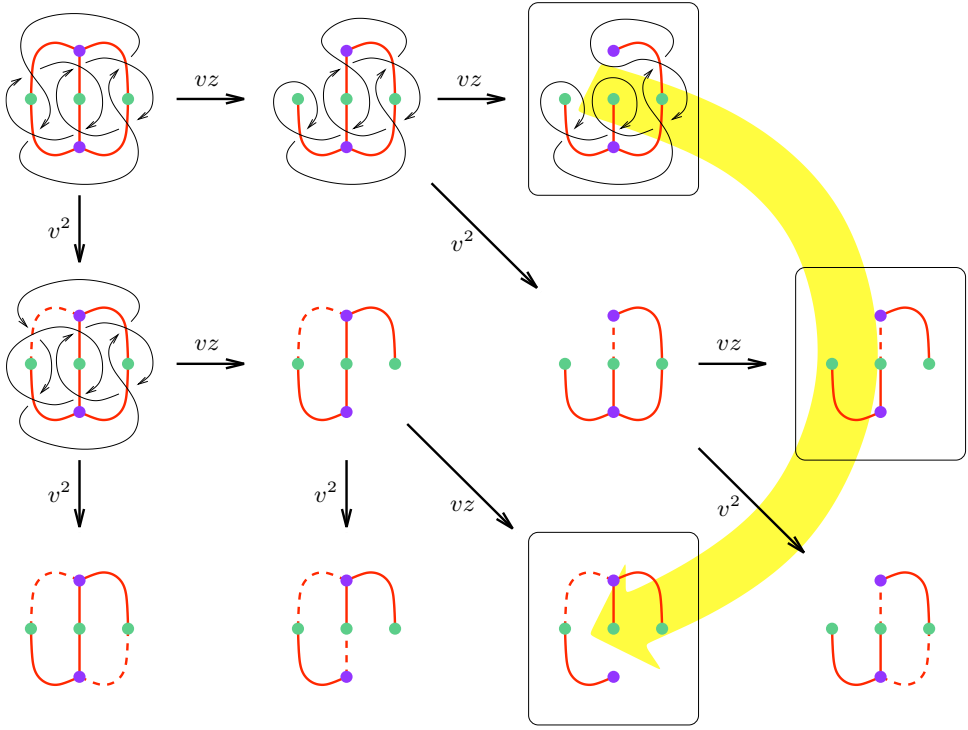} 
   \caption{The tree of subgraphs $\mathscr G$, partially superimposed with the computation tree $\mathscr T$. The semicircular arc indicates the right-to-left order of type I leaves.}
   \label{fig:szamolofa}
\end{figure}

For the complete bipartite graph $G=K_{3,2}$ of Figure \ref{fig:K32}, the link~$L_G$ consists of three fibers of the Hopf fibration. Its diagram has $n=6$ crossings and $s=5$ Seifert circles, so the top of the HOMFLY polynomial is at $z$-exponent $n-s+1=2$. The three type~II leaves of the computation tree $\mathscr T$ of Figure \ref{fig:szamolofa} are two two-component unlinks with the HOMFLY polynomial $(v^{-1}-v)/z$ and a distant union of a Hopf link with an unknot, with $P=\left(\begin{array}{ll}\hspace{8pt}vz&\\+vz^{-1}&-v^3z^{-1}\end{array}\right)\cdot\displaystyle\frac{v^{-1}-v}z$. (These values are well known and easy to check. We write the HOMFLY polynomial in a slightly unconventional way to emphasize its bigraded nature.) Hence
\begin{multline}\label{eq:K32homfly}
P_{L_G}(v,z) \ = \ 
\underbrace{2v^5z\cdot(v^{-1}z^{-1}-vz^{-1})
+v^4\cdot\left(\begin{array}{lll}
\hspace{8pt}1&-v^2&\\
+z^{-2}&-2v^2z^{-2}&+v^4z^{-2}
\end{array}\right)}_{\text{type II leaves}}\\
+\underbrace{(v^2z^2+2v^4z^2)\cdot1}_{\text{type I leaves}}
\ = \ \begin{array}{llll}
&&&\\
v^2z^2&+2v^4z^2&&\\
&+3v^4&-3v^6&\\
&+v^4z^{-2}&-2v^6z^{-2}&+v^8z^{-2}.
\end{array}
\end{multline}
In particular, the top of $P_{L_G}$ is $v^2+2v^4$. Notice how, in accordance with the proof of Theorem \ref{thm:treetop}, contributions to the top came from type I leaves only. The Conway polynomial is $\nabla_{L_G}(z)=P_{L_G}(1,z)=3z^2$ and the Alexander polynomial is $\Delta_{L_G}(t)=\nabla_{L_G}(t^{1/2}-t^{-1/2})=3t-6+3t^{-1}$.
\end{ex}

\begin{rem}
If one color class of $G$ consists of valence $2$ points, that is if $G$ is obtained from a plane graph by placing an extra vertex at the midpoint of each edge, then (as in the previous example) the Conway polynomial $\nabla_{L_G}(z)$ is a single monomial. This follows, for example, from a result of Jaeger \cite{jaeger}.
\end{rem}

\section{Triangulations of the root polytope}\label{sec:rootpoly}

In order to relate our results on arborescences to a different kind of combinatorics and to prove Theorem \ref{thm:triang}, we need to review some of Alexander Postnikov's results from his remarkable paper~\cite{post}. 

Let $G$ be an abstract bipartite graph. That is, we do not assume an embedding of $G$ into the plane. Let us denote the color classes of $G$ with $E$ and $V$. The graph~$G$ may have multiple edges but they do not affect the following construction.

For $e\in E$ and $v\in V$, let $\mathbf e$ and $\mathbf v$, respectively, denote standard generators of $\R^E\oplus\R^V$. Let the \emph{root polytope} of $G$ be
\[Q_G=\conv\{\,\mathbf e+\mathbf v\mid ev\text{ is an edge in }G\,\},\]
where $\conv$ denotes the usual convex hull.

\begin{lem}[{\cite[Lemma 12.5]{post}}]\label{lem:affindep}
Let $G$ be a connected bipartite graph on $s$ vertices.
The dimension of $Q_G$ is $s-2$.
A set of vertices of $Q_G$ is affine independent if and only if the corresponding edges in $G$ form a cycle-free subgraph. In particular, maximal (i.e., 
$(s-2)$-dimensional) simplices in $Q_G$ correspond to spanning trees of $G$. Furthermore, the volumes of such maximal simplices agree.
\end{lem}

\begin{defn}\label{def:triang}
A collection of maximal simplices in $Q_G$ (so that their vertices are also vertices of the root polytope) is a \emph{triangulation} if their union is $Q_G$ and if every two of them intersect in a common face. 
\end{defn}

Studying these triangulations reveals many a subtle phenomenon, see \cite{post}. First let us quote the translation, to subgraph terms, of the second condition of the definition.

\begin{lem}[{\cite[Lemma 12.6]{post}}]\label{lem:compa}
Let $\Gamma_1$ and $\Gamma_2$ be spanning trees in $G$. The following two statements are equivalent.
\begin{enumerate}[label={\rm (\roman*)}]
\item The simplices in $Q_G$ that correspond to the $\Gamma_i$ intersect in a common face.
\item There does not exist a cycle $\varepsilon_1,\varepsilon_2,\ldots,\varepsilon_{2k}$ of edges in $G$, where $k\ge2$, so that all odd-index edges are from $\Gamma_1$ and all even-index edges are from $\Gamma_2$.
\end{enumerate}
\end{lem}

The proofs of the previous two lemmas are short and elementary. From the last assertion in Lemma \ref{lem:affindep}, it follows that each triangulation of $Q_G$ consists of the same number of simplices. Postnikov also expresses that value in terms of $G$. 

\begin{thm}[Postnikov \cite{post}]\label{thm:post}
Let $G$ be a connected bipartite graph with color classes~$E$ and $V$. The number of simplices in each triangulation of the root polytope $Q_G$ is the number of possible valence distributions, taken at elements of $E$, of spanning trees of $G$.
\end{thm}

There is an obvious sense in which $(V,E)$ is a hypergraph \cite{hypertutte} and in that context it is fairly natural to rename (essentially) the valence distributions above as follows.

\begin{defn}
Let $G$ be a connected bipartite graph with color classes $E$ and $V$. A function $\mathbf f\colon E\to\N$ is called a \emph{hypertree} (in the hypergraph~$(V,E)$) if $G$ has a spanning tree with valence $\mathbf f(e)+1$ at each $e\in E$.
\end{defn}

With this, Theorem \ref{thm:post} says that the number of simplices needed to triangulate $Q_G$ is the number of hypertrees in $(V,E)$. Of course, the same can be claimed regarding the `abstract dual' hypergraph $(E,V)$ as well.

\begin{proof}[Proof of Theorem \ref{thm:triang}.]
Even though all the necessary ingredients are included in a previous paper~\cite{hypertutte}, we spell out a proof for completeness and in order to elaborate on some details. As usual, let us denote the color classes in $G$ by $E$ and $V$.

Let us consider a pair $A_1$, $A_2$ of arborescences in $G^*$ (both rooted at $r_0$). Assume that their dual trees $\Gamma_1$, $\Gamma_2$ violate the condition in Lemma \ref{lem:compa}, that is, there exists a cycle $\Phi$ in $G$ composed of edges alternately from $\Gamma_1$ and $\Gamma_2$. Let $\Phi$ bound the disks $U$ and $U'$ in $S^2$. Since elements of $E$ and $V$ alternate along $\Phi$, it easily follows that all edges of $A_1$ that cross $\Phi$ do so from one side to the other, say from $U$ to $U'$. The same is true for edges of $A_2$ crossing $\Phi$, but those travel from $U'$ to $U$. But then, the fact that $r_0$ cannot be in $U$ and in $U'$ simultaneously prevents one of $A_1$ and $A_2$ from being an arborescence: If, say, $r_0\in U$, then $A_2$ is not an arborescence because the startpoints of its edges crossing $\Phi$ (which exist because $A_2$ is connected) cannot be reached by an oriented path from $r_0$.

Now that we have seen that all pairs of simplices resulting from our arborescences satisfy the compatibility condition in Definition \ref{def:triang}, we just have to make sure that there is enough of them so that their union is the entire root polytope. By Theorem~\ref{thm:post}, it suffices to show that all hypertrees in $(V,E)$ are realized by spanning trees in $G$ dual to spanning arborescences in $G^*$. But that is part of the statement of Theorem 10.1 in \cite{hypertutte}.
\end{proof}

\begin{ex}
\begin{figure}[htbp]
\hspace*{-30pt}
\begin{minipage}{.35\linewidth}
\labellist
\tiny
\pinlabel $r_0$ at 77 289
\pinlabel $r_1$ at 257 289
\pinlabel $r_2$ at 437 289
\pinlabel $v_0$ at 346 108
\pinlabel $v_1$ at 346 469
\pinlabel $e_0$ at 166 289
\pinlabel $e_1$ at 346 289
\pinlabel $e_2$ at 526 289
\pinlabel $\kappa$ at 210 120
\endlabellist
\raggedright
\includegraphics[width=\linewidth]{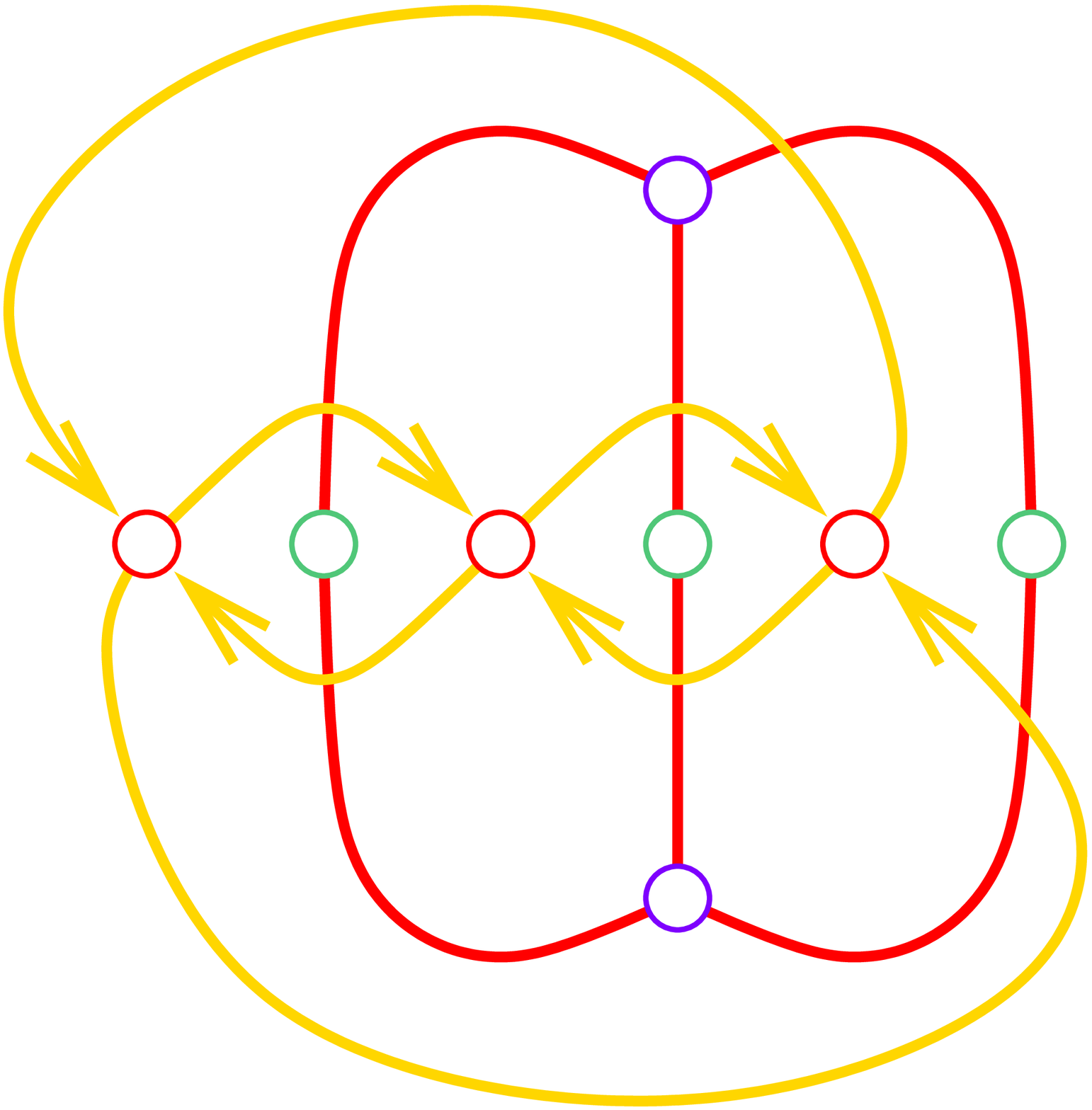}
\end{minipage}
\hspace{.03\linewidth}
\begin{minipage}{.08\linewidth}
\includegraphics[width=\linewidth]{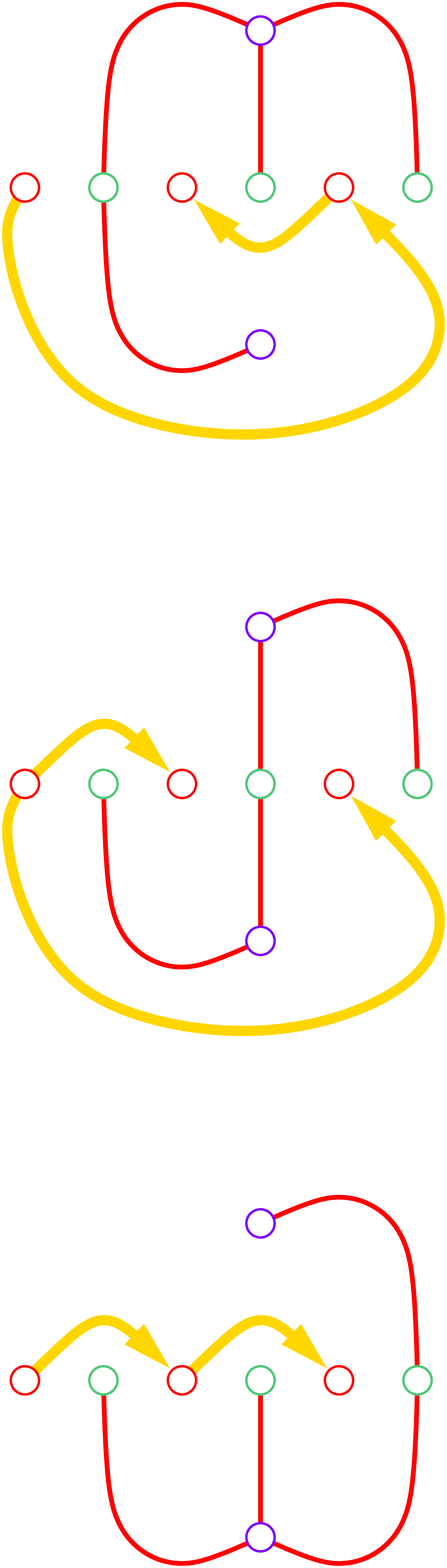}
\end{minipage}
\begin{minipage}{.1\linewidth}
\includegraphics[width=\linewidth]{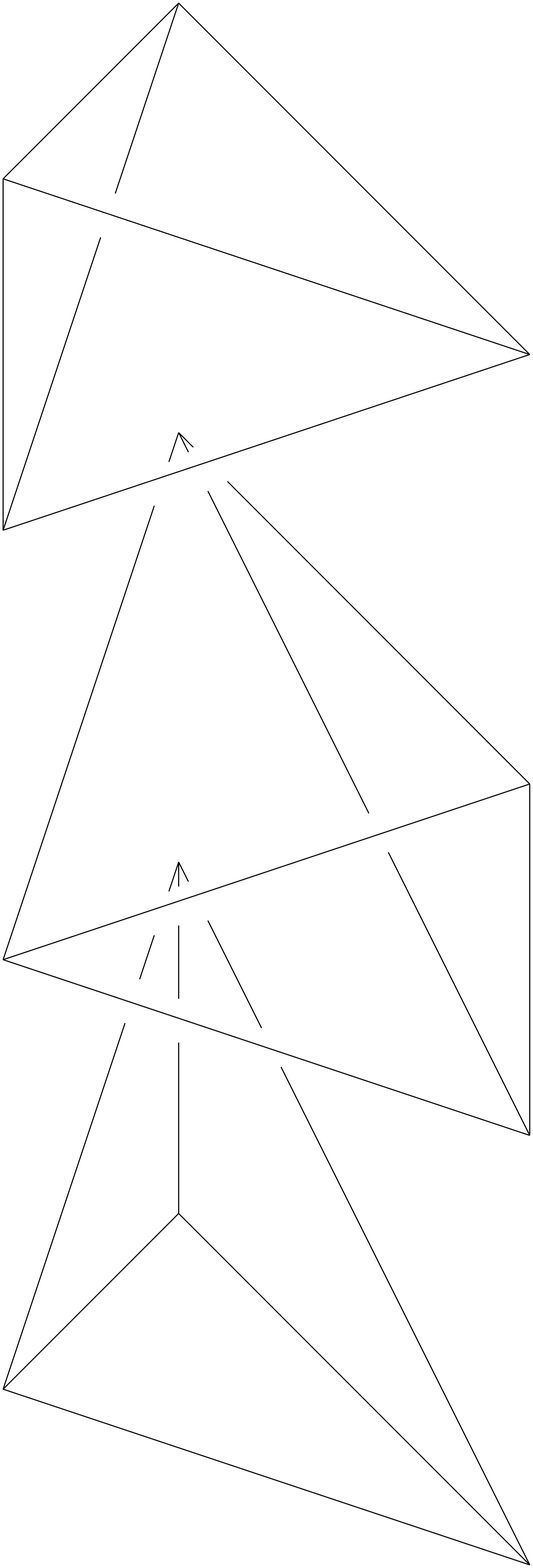}
\end{minipage}
\hspace{.07\linewidth}
\begin{minipage}{.25\linewidth}
\labellist
\tiny
\pinlabel ${\mathbf e}_0+{\mathbf v}_0$ at -65 170
\pinlabel ${\mathbf e}_0+{\mathbf v}_1$ at -65 490
\pinlabel ${\mathbf e}_1+{\mathbf v}_0$ at 560 10
\pinlabel ${\mathbf e}_1+{\mathbf v}_1$ at 560 330
\pinlabel ${\mathbf e}_2+{\mathbf v}_0$ at 240 330
\pinlabel ${\mathbf e}_2+{\mathbf v}_1$ at 240 660
\endlabellist
\includegraphics[width=\linewidth]{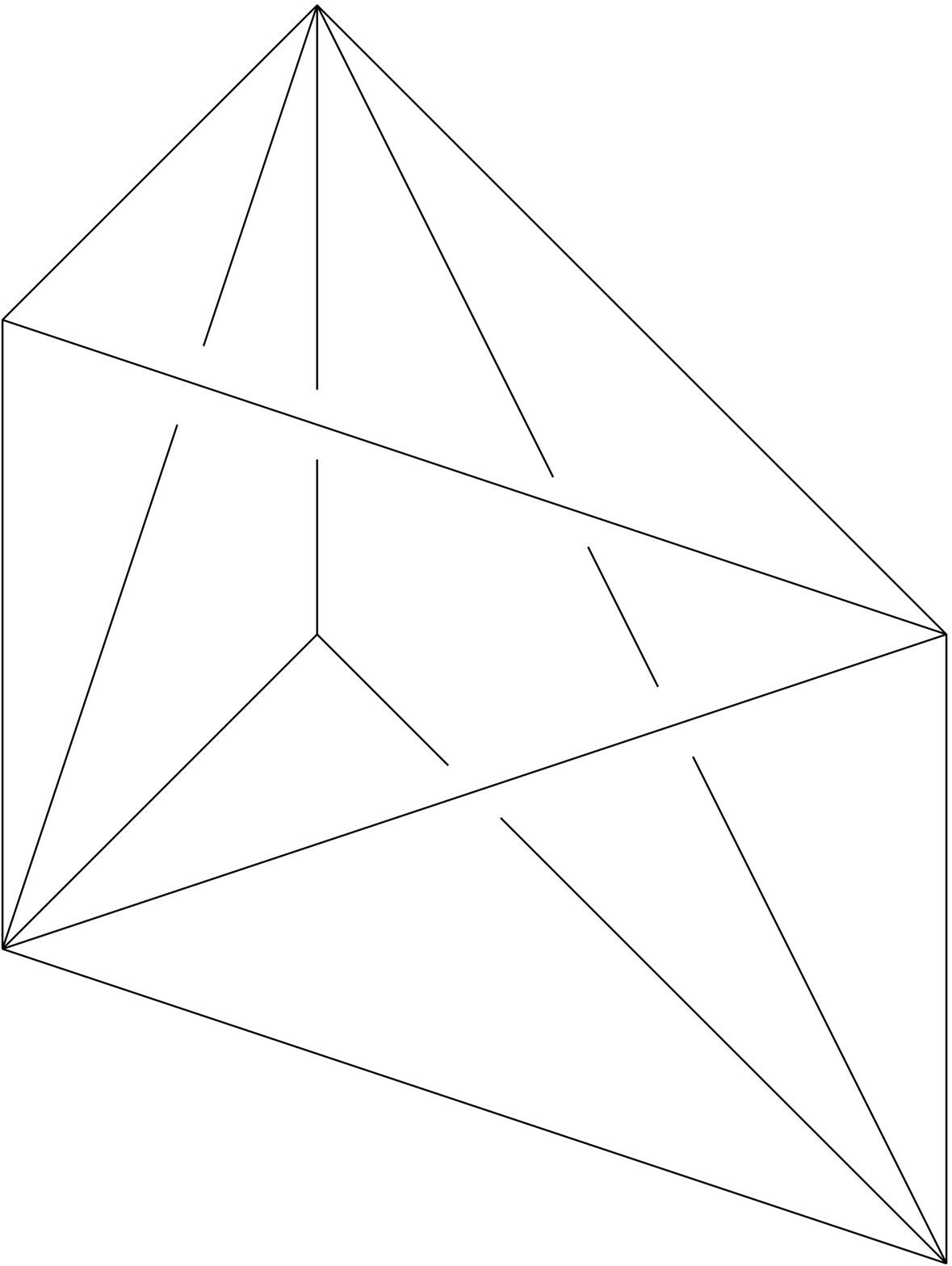}
\end{minipage}
\caption{Left: The graph $G=K_{3,2}$ with color classes $E=\{\,e_0,e_1,e_2\,\}$ and $V=\{\,v_0,v_1\,\}$, and its dual $G^*$. Middle: The spanning arborescences in $G^*$ (relative to the root $r_0$), their dual spanning trees in $G$, and the corresponding maximal simplices in $Q_G$. Right: The three simplices triangulate $Q_G$.}
\label{fig:triang}
\end{figure}

The root polytope of the complete bipartite graph $G=K_{3,2}$ is the product of an interval and a triangle. (In general, $Q_G$ is obtained from the product $\Delta_E\times\Delta_V$ of the $(|E|-1)$-dimensional unit simplex $\Delta_E$ and the $(|V|-1)$-dimensional unit simplex $\Delta_V$ by truncating vertices corresponding to non-edges of $G$.) In Figure~\ref{fig:triang}, we show the triangulation of $Q_G$ corresponding to the arborescences found in Example \ref{ex:arb_tree}. 
\end{ex}

\section{The HOMFLY polynomial and the root polytope}\label{sec:triang}

In this section we prove Theorem \ref{thm:homfly=h}. As explained in the Introduction, the type~I leaves of the binary tree~$\mathscr A$ have a natural order from right to left (see Figure \ref{fig:szamolofa} for an example). The first (smallest) element in the order is the clocked arborescence. 

\begin{thm}\label{thm:shell}
Let $G$ be a connected plane bipartite graph. 
For any choice of root~$r_0$ and adjacent edge $\kappa$ that we use to triangulate the root polytope $Q_G$, the order on the set of maximal simplices induced by the right-to-left order of the corresponding spanning arborescences 
is a shelling order. In particular, the triangulations of $Q_G$ described in Theorem \ref{thm:triang} are shellable. 
\end{thm}

\begin{proof}
As always, let us fix the root $r_0$ and the edge $\kappa$ that are used in the construction of the tree $\mathscr A$. Recall that the vertex set of $Q_G$ is identified with the set~$C$, which in turn can be viewed both as the edge set of $G$ and as the edge set of $G^*$. 

Let $A$ and $B$ be spanning arborescences so that $A<B$, i.e., the unique leaf of $\mathscr A$ that involves $A$ (cf.\ Lemma~\ref{lem:allarboappear}) is to the right of the leaf involving $B$. We have to find a third spanning arborescence $A'$ (allowing for $A'=A$) so that the corresponding simplices $\sigma_A, \sigma_{A'}, \sigma_{B}\subset Q_G$ (with vertex sets identified with $C\setminus A$, $C\setminus A'$, and $C\setminus B$, respectively) satisfy
\begin{enumerate}
\item\label{ichi} $\sigma_B\cap\sigma_{A'}$ is a codimension one face, 
i.e., $A'$ and $B$ differ in exactly one edge
\item\label{ni} $\sigma_B\cap\sigma_A\subset\sigma_{A'}$, that is, $A'\subset A\cup B$
\item\label{san} $\sigma_{A'}$ precedes $\sigma_B$, i.e., $A'<B$ in the right-to-left order.
\end{enumerate}
Let us find the node $(A_0,S_0)$ of $\mathscr A$ that is the last common node along the paths connecting the root to $A$ and $B$, respectively. Let $\delta$ be the augmenting edge of $(A_0,S_0)$. Then $\delta$ is an element of $A$ and it is a skipped edge for $B$. We construct $A'$ by adding $\delta$ to $B$ and removing the edge of $B$ with the same endpoint. By \cite[Lemma~9.8]{hypertutte}, this procedure is well defined and results in a spanning arborescence~$A'$. (Furthermore, $A'$ is the unique spanning arborescence that contains $\delta$ and all but one edge of $B$.) Let us check that $A'$ satisfies our requirements.

The condition \eqref{ichi} is obviously true by construction. The only edge of $A'$ that is not an edge of $B$ is $\delta$; since that is an edge of $A$, \eqref{ni} holds as well. Finally, we claim that the leaf of $\mathscr A$ involving $A'$ is either a descendant of the right descendant~$N=(A_0\cup\{\,\delta\,\}, S_0)$ of $(A_0,S_0)$, or else, the path from the root to $A'$ separates from the path to $N$ so that at some node, the former goes to the right and the latter to the left. This implies \eqref{san} immediately.

Indeed, if the path from the root to $A'$ separated from the path to $N$ when taking a step to the left, then an edge of $A_0\cup\{\,\delta\,\}$ (namely the augmenting edge at the parting of the paths) would be a skipped edge for $A'$. But since $A_0\cup\{\,\delta\,\}\subset A'$, this is impossible and the proof is complete.
\end{proof}

\begin{thm}\label{thm:hfromskipped}
When we compute the $h$-vector of the triangulation in Theorem~\ref{thm:triang} using the shelling of Theorem \ref{thm:shell}, the contribution $c_i$ (cf.\ \eqref{eq:ftoh}) of each simplex $\sigma_i$ is equal to the number of skipped edges for the corresponding type I leaf of $\mathscr A$.
\end{thm}

\begin{proof}
In the proof of Theorem \ref{thm:shell}, a spanning arborescence $A'$ was constructed for any spanning arborescence $B$ and skipped edge $\delta\in S_B$. (Here $S_B$ is the set of skipped edges for the unique leaf of $\mathscr A$ involving $B$, cf.\ Lemma \ref{lem:allarboappear}. Indeed, $A'$ depended on $A$ only through the choice of $\delta$.) Since $\delta$ becomes an edge of $A'$ in the construction, it is clear that different choices of $\delta$ yield different arborescences $A'$. By \eqref{ichi} and \eqref{san} of the proof, we then see that for any $B$, the number $c_B$ of maximal simplices that precede $\sigma_B$ in the shelling order and share a common facet with it, is at least $|S_B|$.

To prove the converse inequality, fix $B$ and let $A$ be a spanning arborescence that precedes $B$ in the right-to-left order so that $A$ only differs from $B$ in one edge. As in the proof of Theorem \ref{thm:shell}, let $(A_0,S_0)$ be the last common node of $\mathscr A$ along the paths from the root of $\mathscr A$ to $A$ and $B$, respectively. Then, since $A<B$, it is clear that the path toward $A$ passes through the right descendant of $(A_0,S_0)$ and the path toward $B$ passes through the left descendant. Therefore, if $\delta$ is the augmenting edge of $(A_0,S_0)$, then $\delta\in A$ but $\delta$ is a skipped edge for $B$. As $A$ and $B$ do not otherwise differ, it is clear that $A$ coincides with the arborescence $A'$ (for the $\delta$ we have just chosen) of the previous paragraph.
\end{proof}

\begin{proof}[Proof of Theorem \ref{thm:homfly=h}]
Using the description \eqref{eq:ftoh} of the $h$-vector and the fact that the dimension of $Q_G$ is $d=|E|+|V|-2=s-2$, we have
\begin{multline*}
\text{top of }P_{L_G}=v^{n-s+1}\sum_{\text{type I leaves}} v^{2k}=v^{n-s+1}\sum_{\text{maximal simplices}} v^{2c_i}\\
=v^{n+s-1}\sum_{\text{maximal simplices}} v^{-2(s-1)+2c_i}=v^{n+s-1}h(v^{-2}),
\end{multline*}
where the first equation is the statement of Theorem \ref{thm:treetop} and the second follows from Theorem \ref{thm:hfromskipped}.
\end{proof}

\begin{ex}
In Figure \ref{fig:szamolofa}, we indicated the right-to-left order of the type I leaves of the computation tree for the graph $K_{3,2}$. In the middle panel of Figure \ref{fig:triang}, we find the three corresponding simplices arranged from bottom (smallest) to top. That is a shelling order for the triangulation so that $c_1=0$ and $c_2=c_3=1$. Hence the $h$-vector is $h(x)=x^4+2x^3$. This, when compared to \eqref{eq:K32homfly}, confirms Theorem~\ref{thm:homfly=h} in this case.
\end{ex}

\section{The HOMFLY polynomial and parking functions}\label{sec:park}

The goal of this section is to prove Theorem \ref{thm:homfly=park}.
The following definition is due to Postnikov and Shapiro \cite{ps}. The only modification we made was to replace (relative) out-degree with in-degree, which is equivalent to an overall reversal of orientation in the directed graph.

\begin{defn}\label{def:parking}
Let $J=(R,C)$ be a directed graph with root $r_0\in R$. For a non-empty subset $R'\subset R\setminus\{r_0\}$ and $r\in R'$, define the \emph{relative in-degree} $\deg_{R'}(r)$ of $r$ as the number of edges in $C$ with endpoint $r$ and startpoint outside of $R'$. 

A function~$\pi\colon R\setminus\{r_0\}\to\N$ is called a \emph{parking function} of $J$ with respect to $r_0$ if any non-empty subset $R'\subset R\setminus\{r_0\}$ contains a vertex $r$ so that $\pi(r)<\deg_{R'}(r)$.

Let us denote the set of parking functions with $\Pi$ and introduce the polynomial 
\[p(u)=\sum_{\pi\in \Pi}u^{\left(\sum\limits_{r\in R\setminus\{r_0\}}\pi(r)\right)},\]
which we call the \emph{parking function enumerator}.
\end{defn}

\begin{ex}
We determine parking functions for the directed graph $G^*$ of our running example (see the left panel of Figure \ref{fig:triang} for notation). If $r_0$ plays the role of root, then the values $\pi(r_1),\pi(r_2)$ are subject to three conditions. The single element of $\{r_i\}$ has relative in-degree $2$ ($i=1,2$), which implies $0\le\pi(r_i)\le1$. Both elements of $\{\,r_1,r_2\,\}$ have relative in-degree $1$ and hence for any parking function $\pi$, one of $\pi(r_1),\pi(r_2)$ has to be $0$.

There are three solutions for $(\pi(r_1),\pi(r_2))$: $(0,0)$, $(1,0)$, and $(0,1)$. Therefore the parking function enumerator for $G^*$ is $p(u)=1+2u$ and, by comparing this to \eqref{eq:K32homfly}, we see that Theorem \ref{thm:homfly=park} holds in this case.
\end{ex}

It is well known that parking functions of $J$ with respect to $r_0$ are in a one-to-one correspondence with spanning arborescences of $J$ rooted at $r_0$ \cite{ps}. Several bijections have been given between the two sets \cite{jonevevan}. In the next theorem we describe another identification that is special to the case when $J=G^*$ for a plane bipartite graph~$G$. 
Recall that the tree $\mathscr A$ constructed in Section \ref{sec:arbo} is such that for any choice of auxiliary data $r_0,\kappa$, and for any spanning arborescence $A$ of $G^*$ rooted at $r_0$, there is a unique set $S$ of skipped edges so that $(A,S)$ becomes a type I leaf of $\mathscr A$.

\begin{thm}\label{thm:park}
For a connected plane bipartite graph $G$, root $r_0\in G^*$, and adjacent edge~$\kappa\in G$, let $(A,S)$ be a type I leaf of the tree of arborescences $\mathscr A$. For every vertex~$r\ne r_0$ of $G^*$, let $\pi(r)$ be the number of edges in $S$ that point to $r$. Then $\pi$ is a parking function with respect to $r_0$.

Furthermore, every parking function arises from a unique spanning arborescence in the manner described above.
\end{thm}

\begin{proof}
Let $\pi$ be derived from $A$ as in the Theorem. Fix a non-empty set~$R'\subset R\setminus\{r_0\}$. Regarding the unique path in $\mathscr A$ from the root $(\varnothing,\varnothing)$ to $(A,S)$, let $(A',S')$ be the first node along the path so that the root component of $A'$ has a vertex, say $r$, from $R'$. It is obvious that $A'\subset A$ and $S'\subset S$. All edges of $S$ that end at $r$ actually belong to $S'$ because once $r$ is in the root component, no edges that end there can enlarge the arborescence and therefore they will never serve as augmenting edges. It is also clear that all edges in $S'$ ending at $r$, as well as the edge of $A'$ that ends at $r$, have their startpoint in $R\setminus R'$. Hence $\pi(r)<\deg_{R'}(r)$, which proves our first claim.

Let us now fix an arbitrary parking function $\pi\colon R\setminus\{r_0\}\to\N$. We will construct the corresponding spanning arborescence by finding the path in $\mathscr A$ that leads to it from the root. The root $(\varnothing,\varnothing)$ is obviously such that the number of skipped edges ending at any vertex $r\in R\setminus\{r_0\}$ is at most $\pi(r)$. Suppose that we have already built a path in $\mathscr A$ ending at the node $(A,S)$ that also has the property that 
\begin{equation}\label{eq:kellmeg}
\text{each vertex }r\in R\setminus\{r_0\}\text{ has at most }\pi(r)\text{ elements of }S\text{ pointing to it.}
\end{equation}

Let the augmenting edge of $(A,S)$ be $\delta=(q,r)$. If $\pi(r)$ is strictly more than the number of edges in $S$ pointing to $r$, then we pass to the left descendant $(A,S\cup\{\delta\})$ of $(A,S)$. If $\pi(r)$ equals the number of elements of $S$ ending at $r$, then we pass to the right descendant $(A\cup\{\delta\})$. In either case, the new node of $\mathscr A$ has the property~\eqref{eq:kellmeg}, so we may continue along our path until we arrive at a leaf $(A_\pi,S_\pi)$ in $\mathscr A$.

We claim that $(A_\pi,S_\pi)$ cannot be a type II leaf. Indeed, assume that the set $R'$ of isolated points of $A_\pi$ is non-empty. Since $\pi$ is a parking function, there exists a vertex $r\in R'$ so that the number of edges starting in the root component and ending at $r$ is more than $\pi(r)$. But because all of those edges belong to $S_\pi$, this contradicts \eqref{eq:kellmeg}.

By the previous paragraph, $A_\pi$ is a spanning arborescence. It is clear from the construction that its associated parking function is $\pi$: for each $r\in R\setminus\{r_0\}$, we did select an edge pointing to $r$ into $A_\pi$ and we did so when exactly $\pi(r)$ skipped edges lead to $r$.

Finally, we argue that spanning arborescences $A'\ne A_\pi$ may not induce $\pi$. Let the paths in $\mathscr A$ that lead from the root to $(A_\pi,S_\pi)$ and to $(A',S')$, respectively (for the appropriate sets $S_\pi$, $S'$), part ways at the node $(A,S)$. Let $(A,S)$ have the augmenting edge $\delta=(q,r)$. If our earlier choice (leading to $(A_\pi,S_\pi)$) was the left descendant of $(A,S)$ then the right descendant, along with all its subsequent descendants including $(A',S')$, is such that the number of skipped edges pointing to $r$ is less than $\pi(r)$ (they ``reach $r$ too soon''). On the other hand, if earlier we chose the right descendant, then along the branch of $\mathscr A$ corresponding to the left descendant, we have more than $\pi(r)$ skipped edges ending at $r$ (i.e., those arborescences ``reach $r$ too late''). This completes the proof.
\end{proof}


\begin{proof}[Proof of Theorem \ref{thm:homfly=park}]
Since every skipped edge for the spanning arborescence $A$ has its unique endpoint in $R\setminus\{r_0\}$, the number of skipped edges for $A$ is the same as the sum of the values of the corresponding parking function. Thus by the definition of the parking function enumerator (cf.\ \eqref{eq:index}, \eqref{eq:park}), we have
\begin{equation*}
\text{top of }P_{L_G}=v^{n-s+1}\sum_{\text{type I leaves}} v^{2k}=
v^{n-s+1}\sum_{\text{parking functions }}v^{2i(\pi)}
=v^{n-s+1}p(v^2).
\end{equation*}
I.e., Theorem \ref{thm:homfly=park} is a direct consequence of Theorems \ref{thm:treetop} and \ref{thm:park}.
\end{proof}

Finally, we note that in the case when the directed graph $J$ is derived from a connected undirected graph~$K$ by replacing each edge with a pair of oppositely oriented edges between the same two vertices, the parking function enumerator $p_J$ is related to the Tutte polynomial~$T_K$ via the formula
\[p_J(u)=u^{b_1(K)}T_K(1,1/u),\]
where $b_1(K)$ is the first Betti number of $K$ (viewed as a one-dimensional complex). To see this, one first has to note the (quite direct) connection \cite{ps} between parking functions and critical configurations of the abelian sandpile model (a.k.a.\ chip firing game), and then apply a formula of Merino L\'opez \cite{merino}.

In particular, when the plane bipartite graph $G$ is such that one of its color classes, say $E$, contains only degree $2$ vertices (i.e., when the pair $(V,E)$ can be viewed as a graph $K^*$ with planar dual $K$, which induces the directed graph $J=G^*$), then the duality formula of the Tutte polynomial implies
\[p_{G^*}(u)=u^{|V|-1}T_{K^*}(1/u,1)=I_{(V,E)}(u),\]
where $I$ is the interior polynomial \cite{hypertutte} of the (hyper)graph $K^*=(V,E)$. Thus, Theorem~\ref{thm:homfly=park} represents progress toward the Conjecture in the Introduction of \cite{jkr}.

\section{An improvement on Morton's inequality}\label{hitoshi}

This section adds the last remaining piece to establish our results.
Namely, our present goal is to prove that type II leaves in the computation tree of Section \ref{sec:comptree} do not affect the top of the HOMFLY polynomial. We have seen how this boils down to Proposition \ref{biglemma}, an important technical result that slightly strengthens Morton's inequality for a specific kind of link diagram. 

First, we establish several lemmas on oriented curves immersed in the plane. We assume these immersions to be generic, i.e., to have no self-tangencies or triple points. Since that is exactly the class of curves that we get if we forget the crossing information in an oriented link diagram, we will refer to our immersed curves as \emph{link projections}. It is useful for us to study them because the numbers of crossings and of Seifert circles in a link diagram (and thus the upper bound that we are seeking in Proposition \ref{biglemma}) only depend on the corresponding link projection. We may apply Reidemeister moves in this context as well.

\begin{defn}\label{def:reidmoves}
We introduce the following terms for certain isotopies of link projections.\footnote{A note on terminology: `b' moves would be the inverses of the `a' moves below, but those will not play a role in our treatment.} 
\begin{itemize}
\item
A \emph{Reidemeister I-a move} removes a kink of a link projection (see Figure~\ref{fig:RI}).
\begin{figure}[h]
  \includegraphics[width=2.5in]{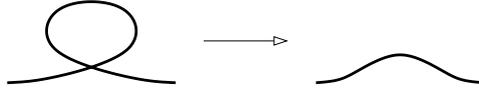}
  \caption{Reidemeister I-a move}
  \label{fig:RI}
\end{figure}
Here, either orientation is allowed.
\item
We call the move shown in Figure~\ref{fig:cyclic_RIIa} (Figure~\ref{fig:noncyclic_RIIa}, respectively) a \emph{cyclic Reidemeister II-a move} (\emph{noncyclic Reidemeister II-a move}, respectively).
\begin{figure}[h]
  \includegraphics[width=3in]{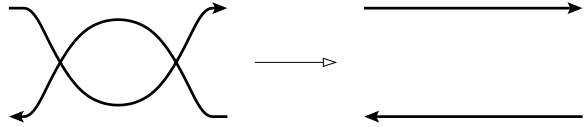}
  \caption{Cyclic Reidemeister II-a move}
  \label{fig:cyclic_RIIa}
\end{figure}
\begin{figure}[h]
  \includegraphics[width=3in]{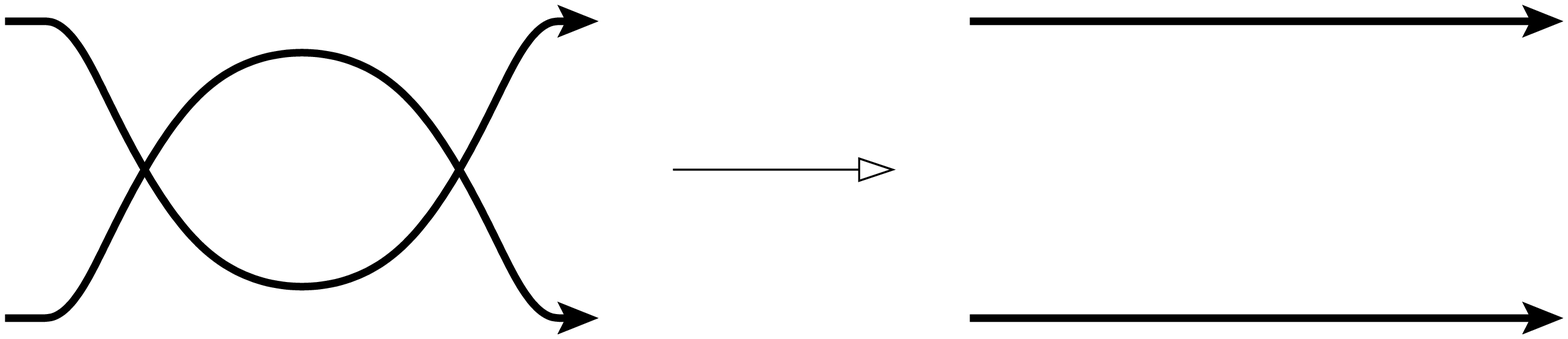}
  \caption{Noncyclic Reidemeister II-a move}
  \label{fig:noncyclic_RIIa}
\end{figure}
\item
The isotopy shown in Figure~\ref{fig:global_noncyclic_RIIa} is called a \emph{global noncyclic Reidemeister II-a move}. Here, the shaded region may contain arcs from the link projection. A noncyclic Reidemeister II-a move is a special case of a global noncyclic Reidemeister II-a move.
\begin{figure}[h]
  \includegraphics[width=3in]{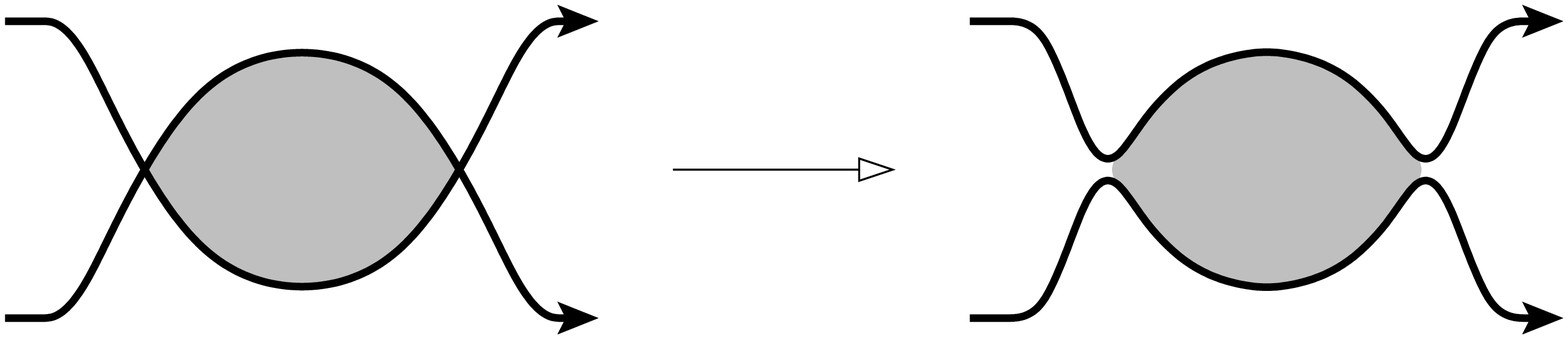}
  \caption{Global noncyclic Reidemeister II-a move}
  \label{fig:global_noncyclic_RIIa}
\end{figure}
\item
We call the move shown in Figure~\ref{fig:noncyclic_RIII} a \emph{noncyclic Reidemeister III move}.
\begin{figure}[h!]
  \includegraphics[width=2.5in]{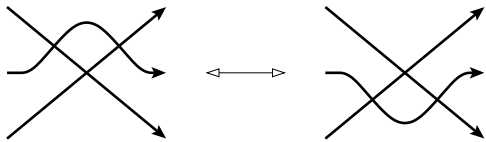}
  \caption{Noncyclic Reidemeister III move}
  \label{fig:noncyclic_RIII}
\end{figure}
\end{itemize}
\end{defn}

For a link projection $D$, let $n(D)$ and $s(D)$ denote the number of crossings and Seifert circles, respectively. Let us analyze how the value $n(D)-s(D)$, which is essentially our desired upper bound, changes under the moves of Definition \ref{def:reidmoves}.

\begin{defn}
If an isotopy (either of a link projection or of a link diagram) reduces $n-s$, then we call it a \emph{good move}.
If the isotopy preserves $n-s$, then we call it a \emph{fair move}.
\end{defn}

\begin{lem}\label{lem:goodmoves}
A Reidemeister I-a move is a fair move, a cyclic Reidemeister II-a move is a good or fair move, a (local or global) noncyclic Reidemeister II-a move is a good move, 
and a noncyclic Reidemeister III move is a fair move.
\end{lem}

\begin{proof}
From the left panel in Figure~\ref{fig:RI_1_Seifert} (where the dashed curves indicate Seifert circles) one sees immediately that a Reidemeister I-a move decreases both the number of Seifert circles and the number of crossings by one. So a Reidemeister I-a move is a fair move.

\begin{figure}[h]
  \includegraphics[width=3in]{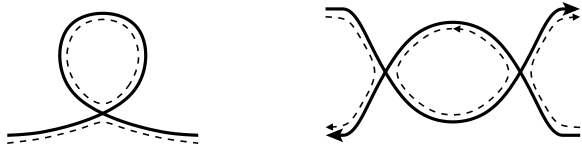}
  \caption{Seifert circles before a 
  I-a and a cyclic II-a move.}
  \label{fig:RI_1_Seifert}
\end{figure}

Let us consider the case of a cyclic Reidemeister II-a move. Before the move, either three or two Seifert circles are adjacent to the two crossings involved in the move. (In Figure~\ref{fig:RI_1_Seifert}, right panel, the Seifert circle leaving at the top-right point may or may not re-emerge at the top-left point.)
In the first case, the move decreases the number of Seifert circles by two, and in the second case, the number of Seifert circles remains the same.
Since the number of crossings is reduced by two, a cyclic Reidemeister II-a move is either a good or a fair move.

Under a (global) noncyclic Reidemeister II-a move, the configuration of Seifert circles does not change. As the number of crossings gets reduced by two, a (global) noncyclic Reidemeister II-a move is a good move.

Since neither the number of Seifert circles nor the number of crossings changes under a noncyclic Reidemeister III move, it is a fair move.
\end{proof}

If one of the good moves above is applied to a link diagram, then the value of $n-s$ actually drops by $2$. Hence if a link diagram admits a good Reidemeister move (or in fact, any good move), then (by the invariance of the HOMFLY polynomial) Theorem \ref{thm:morton} implies Proposition \ref{biglemma} for that case. It is also easy to see the following.

\begin{lem}\label{lem:fair}
If we are able to apply a fair or good move to a link diagram and the estimate of Proposition \ref{biglemma} holds after the move, then it also holds before the move.
\end{lem}


However for now, we are still working in the category of link projections. As our next intermediate step, we will establish a way of handling certain self-intersections. By \emph{emptying} a region of a link projection (or diagram), we mean an isotopy 
that leaves the boundary arcs of the region fixed so that at the end, the interior of the region is disjoint from the projection/diagram.

\begin{lem}\label{lem:monogon}
Any monogon in a link projection can be emptied by a finite sequence of Reidemeister I-a moves, 
cyclic Reidemeister II-a moves, (global) noncyclic Reidemeister II-a moves, and noncyclic Reidemeister III moves.
\end{lem}

\begin{proof}
By starting with an innermost one, we may assume that the monogon contains no other monogons.
In other words, we assume that no arc inside the monogon intersects itself.
If an arc makes a cyclic (noncyclic, respectively) bigon with the monogon, then let us call it a cyclic (noncyclic, respectively) arc.

We proceed by induction on the number of crossings inside and on the boundary of the monogon. The monogon is empty if and only if this number is $1$. Our goal is to show that if a monogon has at least one (directed, non-self-intersecting) arc crossing it, then a sequence of the specified moves can be performed to reduce the number of crossings.

If a global noncyclic Reidemeister II-a move is possible, then we can achieve our goal immediately. Strictly speaking, if one of the arcs shown in Figure \ref{fig:global_noncyclic_RIIa} is part of the monogon's boundary, then the global noncyclic Reidemeister II-a move has to be followed by an isotopy of the plane (i.e., one that does not change the combinatorics of the link projection) to restore the monogon to its original position. Even more strictly speaking, the move itself should be carried out so that the monogon arc stays fixed throughout. It is easy to see that this can be done. For the sake of brevity, we omit this technicality later on.

Hence from now on, we assume that any bigon formed by our arcs (including the monogon's boundary) is cyclic. In particular, all arcs inside the monogon will be assumed cyclic.




Let $\gamma_0$ be the arc in the monogon whose exit point occurs last along the boundary. (We may assume without loss of generality that the monogon is clockwise oriented as in Figure~\ref{fig:monogon}, so that $\gamma_0$ is the arc with the leftmost exit point.) 
\begin{figure}[h]
\labellist
\pinlabel $\gamma_0$ at 70 190
\pinlabel $B_0$ at 400 120
\pinlabel $\alpha_0$ at 560 80
\endlabellist
  \includegraphics[width=3in]{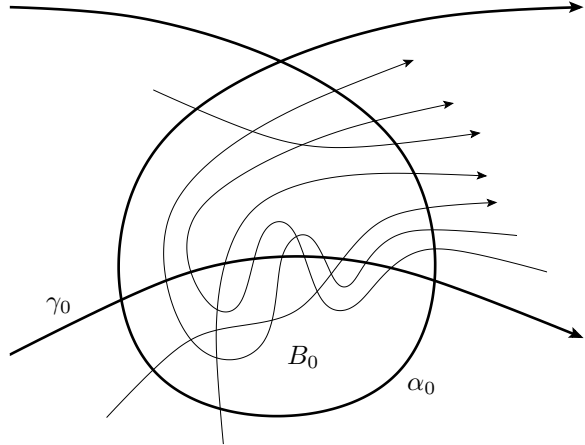}
  \caption{Monogon with cyclic arcs.}
  \label{fig:monogon}
\end{figure}
Let $B_0$ be the bigon determined by $\gamma_0$ and the monogon. Let us also put $\alpha_0=\partial B_0\setminus\gamma_0$ for the monogon arc that bounds $B_0$.

For the remaining part of the proof, we will consider bigons $B$ inside our monogon with their two boundary arcs designated as \emph{upper} and \emph{lower}. For the bigon $B_0$ of the previous paragraph, $\alpha_0$ is the lower arc and $\gamma_0$ is the upper one. By our choice of $\gamma_0$, we see that $B_0$ has the property that
\begin{equation}\label{eq:nbe}
\text{no arc exits the bigon through its lower boundary arc.}
\end{equation}
We will use a recursive procedure to define a nested sequence of bigons $B_0,B_1,\ldots$, all of which will satisfy property \eqref{eq:nbe}. All bigons in the sequence will of course be cyclic and have their boundary oriented the same way (i.e., clockwise).

Keeping in mind our assumption that all bigons are cyclic, it is easy to see that if a bigon has property \eqref{eq:nbe}, then there can only be two kinds of arcs crossing it. The two cases are depicted in Figure \ref{fig:cutbite}, and will be referred to as \emph{cutting} and \emph{biting} arcs, respectively. 

\begin{figure}[htb]
   \centering
   \includegraphics[width=2in]{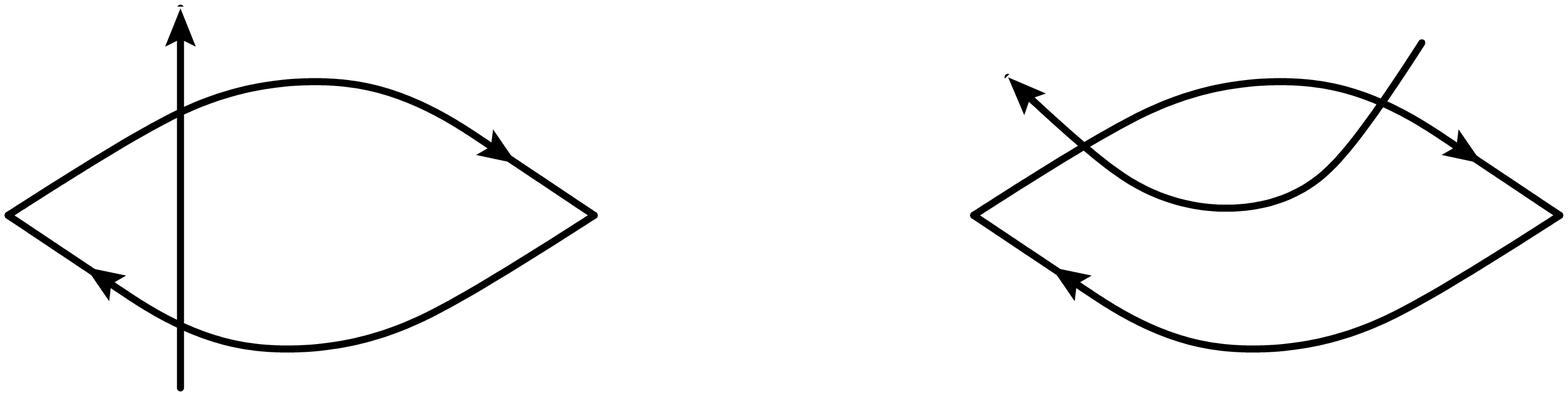} 
   \caption{Cutting (left) and biting (right) arcs in a bigon.}
   \label{fig:cutbite}
\end{figure}

Let now $B_i$ be a bigon with upper arc $\gamma_i$ and lower arc $\alpha_i$ so that property \eqref{eq:nbe} holds. Assume that $B_i$ has a biting arc. Let $\alpha_{i+1}$ denote the biting arc whose exit point from $B_i$ is leftmost along $\gamma_i$. Let $B_i'$ be the bigon formed by $\gamma_i$ and $\alpha_{i+1}$, see Figure \ref{fig:wheredoesitgo}. If no arc (inside $B_i$)  crosses $\alpha_{i+1}$ downward, then $B_i'$, bounded by lower arc $\alpha_{i+1}$ and upper arc~$\gamma_{i+1}=\gamma_i$ (the latter appropriately shortened) has property~\eqref{eq:nbe} and we denote it by $B_{i+1}=B_i'$. Otherwise, let $p_i$ denote the leftmost point along $\alpha_{i+1}$ where an arc~$\gamma_{i+1}$ exits $B_i'$. 

\begin{figure}[htb]
\labellist
\small
\pinlabel $\alpha_i$ at 270 10
\pinlabel $\gamma_i$ at 160 170
\pinlabel $\alpha_{i+1}$ at 280 180
\pinlabel $p_i$ at 130 50
\pinlabel $q_i$ at 200 70
\pinlabel $\searrow$ at 90 130
\pinlabel $B_i'$ at 40 180
\pinlabel $\nearrow$ at 30 50
\pinlabel $B_i$ at -10 10
\pinlabel $\alpha_i$ at 650 10
\pinlabel $\gamma_i$ at 480 160
\pinlabel $\alpha_{i+1}$ at 680 160
\pinlabel $p_i$ at 550 50
\pinlabel $q_i$ at 570 170
\pinlabel $r_i$ at 370 95
\pinlabel $\alpha_i$ at 1030 10
\pinlabel $\gamma_i$ at 1040 130
\pinlabel $\alpha_{i+1}$ at 1040 180
\pinlabel $\alpha_i$ at 1410 10
\pinlabel $\gamma_i$ at 1300 170
\pinlabel $\alpha_{i+1}$ at 1170 170
\pinlabel $\alpha_i$ at 1790 10
\pinlabel $\gamma_i$ at 1680 170
\pinlabel $\alpha_{i+1}$ at 1800 180
\pinlabel $p_i$ at 1710 60
\pinlabel $q_i$ at 1610 60
\endlabellist
   \centering
   \includegraphics[width=\linewidth]{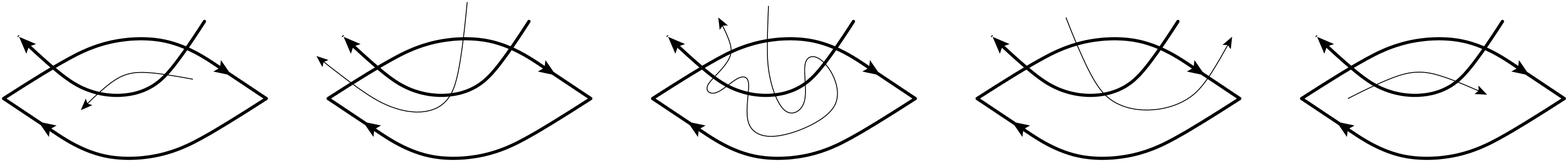} 
   \caption{Hypothetical positions of the arc $\gamma_{i+1}$. Only the last option can be reconciled with our assumptions.}
   \label{fig:wheredoesitgo}
\end{figure}

We claim that the portion of $\gamma_{i+1}$ in $B_i'$ may only be arranged as in the rightmost panel of Figure~\ref{fig:wheredoesitgo}. The reasons for this are as follows. Let $q_i$ denote the point where $\gamma_{i+1}$ enters $B_i'$ for the last time before leaving at $p_i$.
\begin{itemize}
\item The point $q_i$ may not be on $\alpha_{i+1}$ and to the right of $p_i$ because then $\alpha_{i+1}$ and $\gamma_{i+1}$ would form a noncyclic bigon.
\item Assuming that $q_i$ is on $\gamma_i$, we consider the point $r_i$ where $\gamma_{i+1}$ leaves $B_i$ for the first time after passing through $p_i$. By property \eqref{eq:nbe}, $r_i$ is on $\gamma_i$.
\item The point $r_i$ may not be to the left of $B_i'$ because of the way that $\alpha_{i+1}$ was chosen.
\item The point $r_i$ may not be on $\partial B_i'$ and to the left of $q_i$ because in order to get there, $\gamma_{i+1}$ would have to form a noncyclic bigon with $\alpha_{i+1}$.
\item The point $r_i$ may not be to the right of $q_i$ because then $\gamma_{i+1}$ and $\gamma_i$ would form a noncyclic bigon.
\end{itemize}
In this case, then, we let the bigon $B_{i+1}$ be bounded by the lower arc $\alpha_{i+1}$ and the upper arc $\gamma_{i+1}$. Because of the way $p_i$ was chosen, $B_{i+1}$ satisfies property \eqref{eq:nbe}.

We continue constructing the sequence $B_0, B_1,\ldots$ until we run out of bigons. I.e., the last member $\tilde B$ of the sequence has no biting arcs. 
If $\tilde B$ is empty, we can reduce the number of crossings in the monogon by a cyclic Reidemeister II-a move. Otherwise, take the (cutting) arc in $\tilde B$ whose starting point (along the lower part of the boundary) is leftmost.
It cuts $\tilde B$ into two triangles, one of which (on the left) is noncyclic.
By Lemma~\ref{lem:Delta} below, we can find an empty (and, because of the way cutting arcs are oriented, necessarily noncyclic) triangle, adjacent to the upper part of $\partial\tilde B$, inside this triangle.
Let us remove that empty triangle from $\tilde B$ by a noncyclic Reidemeister~III move.
Thus we have reduced the number of crossings inside $\tilde B$ and then by an obvious induction argument we can assume that the cutting arcs in $\tilde B$ do not intersect each other inside the bigon.
Now we can empty $\tilde B$ by a series of noncyclic Reidemeister III moves and finally apply a cyclic Reidemeister II-a move to remove $\tilde B$ itself. This completes the proof by induction.
\end{proof}

We separated the following Lemma from the previous proof because orientation does not matter for this part.

\begin{lem}\label{lem:Delta}
Let $\Delta$ be a triangle in a link projection with no monogon or bigon inside.
If a side of $\Delta$ has no crossings on it, then each of the other two sides has an adjacent empty triangle.
\end{lem}

\begin{proof}
We proceed by induction on the the number of arcs in $\Delta$. If there is no arc at all in $\Delta$, then $\Delta$ itself is the required triangle. Suppose that for any triangle with less than $n$ arcs inside which form no monogons or bigons, and so that there is no crossing on one side of the triangle, we can find an empty triangle next to each of the other two sides.

Let us consider now the triangle $ABC$ with $n$ arcs inside and assume that there are no crossings on the edge $AB$.
All $n$ arcs connect the sides $AC$ and $BC$ since there are no bigons inside $ABC$.
Let $A'$ be the crossing on $AC$ that is nearest to $A$ and let $B'$ be the crossing on $BC$ connected to $A'$ by an arc.

If the arc $A'B'$ does not cross any of the other $n-1$ arcs then the triangle~$A'B'C$ contains all those $n-1$ arcs and we may apply the inductive hypothesis to $A'B'C$ to find empty triangles in it adjacent to $A'C$ and to $B'C$.
Otherwise, let $B''$ be the crossing on $A'B'$ that is nearest to $A'$ and let $C'$ be the crossing formed by $A'C$ and the arc passing through $B''$. (The point $C'$ does occur on $A'C$ because of the way $A'$ was chosen.) 
Then the triangle $A'B''C'$ (with empty side $A'B''$) satisfies the assumption in the inductive hypothesis.
Therefore we can find an empty triangle adjacent to the side~$A'C'$.
The same argument shows that there is an empty triangle adjacent to the side $BC$ as well.
\end{proof}

So far in this paper, we treated skein computation trees from the point of view of the Seifert graph. There is however an older approach, based on the notion of a descending diagram \cite{cromwell}. We will borrow some ideas from that context to prove
an equivalent version of Proposition \ref{biglemma}, which is the main result of this section.

\vspace{-.7cm}

\begin{figure}[htbp]
  \includegraphics[bb=0 100 700 800,width=3in]{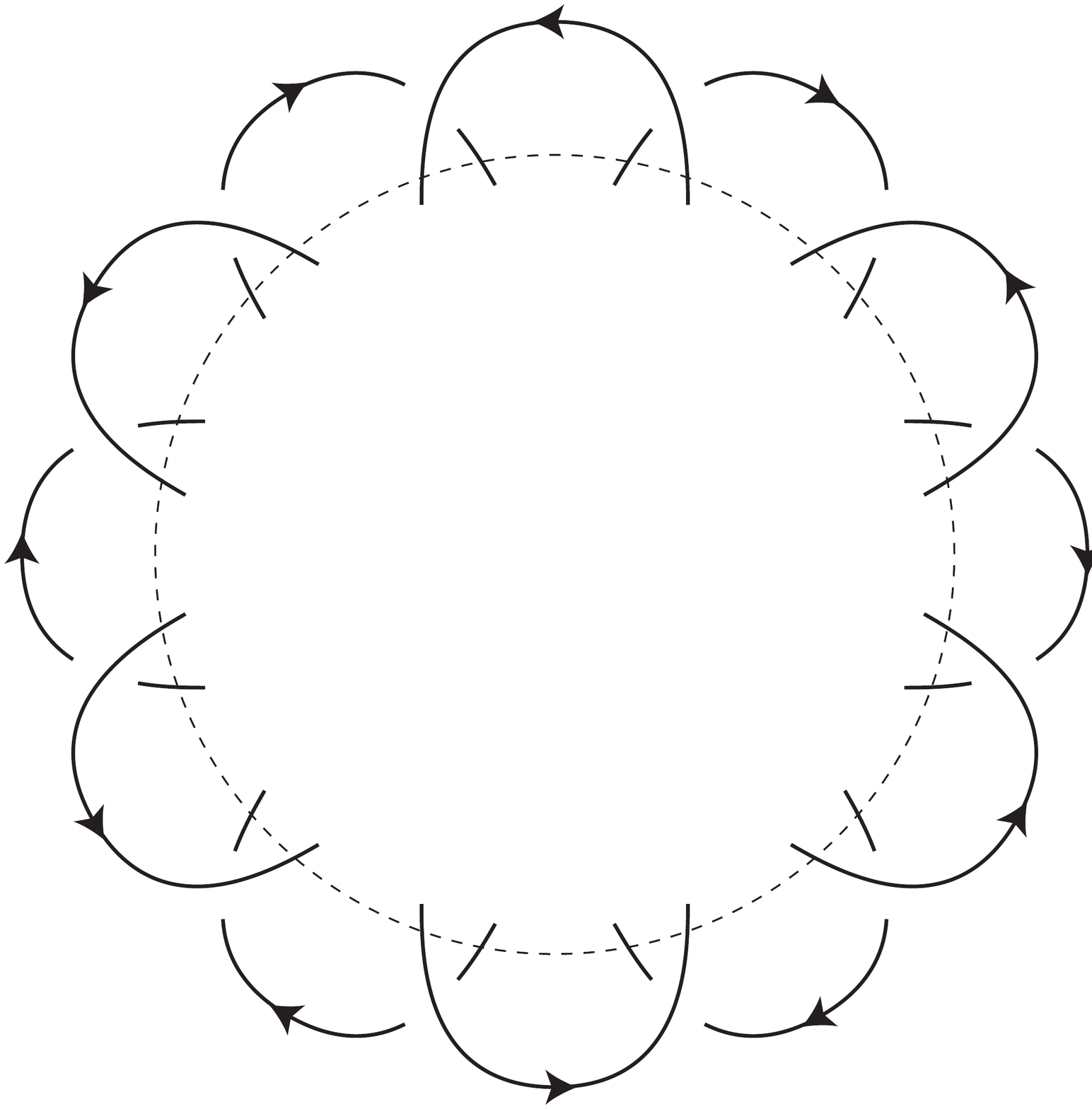}
  \caption{An alternating contour with $k=6$ outer over-arcs and outer under-arcs.}
  \label{fig:outer_arcs}
\end{figure}

\vspace{-.3cm}

\begin{thm}\label{prop:main}
Let $D$ be a link diagram containing the part shown in Figure~\ref{fig:outer_arcs}, and $L$ the associated link. That is, outside of the dashed circle, there is no piece of $D$ other than the $2k$ arcs shown, where $k\ge1$. Then Morton's inequality 
(Theorem~\ref{thm:morton}) is not sharp; in fact, we have
\begin{equation}\label{eq:Morton_nonsharp}
  \maxdeg_z(P_L(v,z))\le n(D)-s(D)-1.
\end{equation}
\end{thm}

\begin{proof}
We will refer to the link diagrams described in the Theorem (and in Proposition~\ref{biglemma}) as having an \emph{alternating contour}. A diagram with an alternating contour is manifestly not alternating, but we are confident that this will not lead to confusion. This notion is related to plane graphs with an alternating contour (cf.\ Section~\ref{sec:comptree}), but it is best to treat the two concepts as separate.

Let us call the $k$ arcs passing over the other arcs appearing in Figure~\ref{fig:outer_arcs} \emph{outer over-arcs}, while the $k$ arcs passing under the outer over-arcs are the \emph{outer under-arcs}. Let us denote the dashed circle shown in the diagram, along which the $4k$ endpoints of all outer arcs lie, with $C$.

We are going to use induction on the number $N$ of crossings that the diagram $D$ has inside $C$. When that number is $0$, we will first argue that outer over-arcs and outer under-arcs may not lie along the same component of $L$. Let us number the $4k$ points of $C\cap D$ counterclockwise around $C$. A quick examination of Figure \ref{fig:outer_arcs} shows that the modulo $4$ remainder class of each point tells exactly if it is a start- or an endpoint and whether of an over-arc or an under-arc. For any arc of $D$ across the interior of $C$ that follows right after an outer over-arc, its endpoint minus its startpoint has to equal $1$ modulo $4$: the difference has to be $1$ or $3$ so that the other arcs across $C$ may complete a crossingless matching of the $4k$ points, but $3$ is ruled out if we consider the prescribed orientations. This means that the arc leaves $C$ at another over-arc and the claim follows.

Now the previous paragraph implies that when $N=0$, the link $L$ is an unlink of at least $2$ components: The components containing the outer over-arcs form an unlink (with a crossingless diagram, no less), the same is true for the components through the under-arcs, and the two unlinks are separated by a copy of $S^2$. Since the Homfly polynomial of a $c$-component unlink is $(v^{-1}z^{-1}-vz^{-1})^{c-1}$, we have $\maxdeg_z\le-1$ in our case. Hence it suffices to show that $-1\le n(D)-s(D)-1=2k-s(D)-1$. But as every Seifert circle of $D$ (in the $N=0$ case) has to pass through the midpoint of at least one outer arc, this follows easily.

As to the inductive step, let us start with a general observation.
Let $(L_+,L_-,L_0)$ be a skein triple and suppose that $L_+$ (or $L_-$) and $L_0$ satisfy the inequality \eqref{eq:Morton_nonsharp}:
\begin{align*}
  \maxdeg_z(P_{L_{\pm}}(v,z))&\le n(D_{\pm})-s(D_{\pm})-1;
  \\
  \maxdeg_z(P_{L_0}(v,z))&\le n(D_0)-s(D_0)-1.
\end{align*}
Since $s(D_+)=s(D_-)=s(D_0)$ and $n(D_+)=n(D_-)=n(D_0)+1$, we have the inequalities
\begin{align*}
  \maxdeg_z(P_{L_{\pm}}(v,z))&\le n(D_{\mp})-s(D_{\mp})-1,
  \\
  \maxdeg_z(P_{L_0}(v,z))&\le n(D_{\mp})-s(D_{\mp})-2.
\end{align*}
On the other hand, as $P_{L_{\mp}}(v,z)=v^{\mp2}P_{L_{\pm}}(v,z)\mp zv^{\mp1}P_{L_0}(v,z)$ from \eqref{skeinstuff}, we see that
\begin{equation*}
  \maxdeg_{z}(P_{L_{\mp}}(v,z))
  \le
  \max\{\,\maxdeg_{z}(P_{\,L_{\pm}}(v,z)),\maxdeg_{z}(P_{L_0}(v,z))+1\,\}.
\end{equation*}
Therefore we have
\begin{equation*}
  \maxdeg_z(P_{L_{\mp}}(v,z))
  \le
  n(D_{\mp})-s(D_{\mp})-1.
\end{equation*}
This means that if $L_0$ and one member of the pair $L_+$, $L_-$ satisfy \eqref{eq:Morton_nonsharp}, then so does the other member.

Let us now fix $k$ and assume that we have a diagram $D$ as in the Theorem, with $N$ crossings inside $C$, as well as that \eqref{eq:Morton_nonsharp} holds whenever the number of crossings inside $C$ is less than $N$.
Because in a skein triple, $L_0$ always has one less crossing than $L_+$ or $L_-$, the observation above means that it suffices to show the following:
\begin{equation}\label{eq:valami}
\parbox[t]{.85\linewidth}{For every diagram $D$ as in the Theorem, with $2k$ outer arcs, it is possible to change some of the $N$ crossings inside $C$ so that \eqref{eq:Morton_nonsharp} holds for the link thus obtained.}
\end{equation}

We wish to apply isotopies and to rely on Lemma \ref{lem:fair} to prove \eqref{eq:valami}. If a Reidemeister move is possible for the corresponding link projection, then it becomes possible for the diagram as well after changing at most one crossing. None of the Reidemeister moves listed in Lemma \ref{lem:goodmoves} and used in the proof of Lemmas \ref{lem:monogon} and \ref{lem:Delta} increases the number of crossings. 

A special note is in order on global noncyclic Reidemeister II-a moves. When we apply such a move to a link diagram, first we change crossings along the boundary of the bigon (including at most one of the vanishing crossings). This is done while the total number of crossings inside $C$ is $N$ or less, so that we can rely on the inductive hypothesis. Then, `during' the isotopy, the number of crossings may temporarily exceed $N$ but that is fine since no crossing change is necessary at those stages.

Thus, with the help of Lemma \ref{lem:monogon} and the fact that a Reidemeister I-a move is fair, \eqref{eq:valami} reduces to 
\begin{equation}\label{eq:megvalami}
\parbox[t]{.85\linewidth}{For every diagram $D$ as in the Theorem with $2k$ outer arcs and at most $N$ crossings inside $C$, so that arcs inside $C$ do not self-intersect, it is possible to change some of the 
crossings inside $C$ so that \eqref{eq:Morton_nonsharp} holds for the resulting link.}
\end{equation}


We keep insisting on making changes inside $C$ only so that our main induction can proceed. We first establish \eqref{eq:megvalami} in two special cases.

\begin{enumerate}[leftmargin=20pt,topsep=5pt,itemsep=5pt,label=\Alph*]
\item\label{caseA}
We assume that there exists an outer over-arc in $D$ whose endpoints are joined by another arc $\alpha$ inside the dashed circle. See Figure \ref{fig:befordit}. By changing some crossings along $\alpha$, let us `float to the top' the unknot component arising from our assumption, and then let us separate it from the rest of the diagram as in the right panel of Figure \ref{fig:befordit}. We denote the resulting diagram with $\widetilde D$.

\begin{figure}[htbp] 
\labellist
\small
\pinlabel $\alpha$ at 220 180
\pinlabel $a$ at 245 350
\pinlabel $b$ at 200 353
\pinlabel $c$ at 127 310
\pinlabel $d$ at 312 313
\pinlabel $a$ at 905 350
\pinlabel $b$ at 845 353
\pinlabel $c$ at 777 310
\pinlabel $d$ at 962 313
\endlabellist
   \centering
   \includegraphics[width=\linewidth]{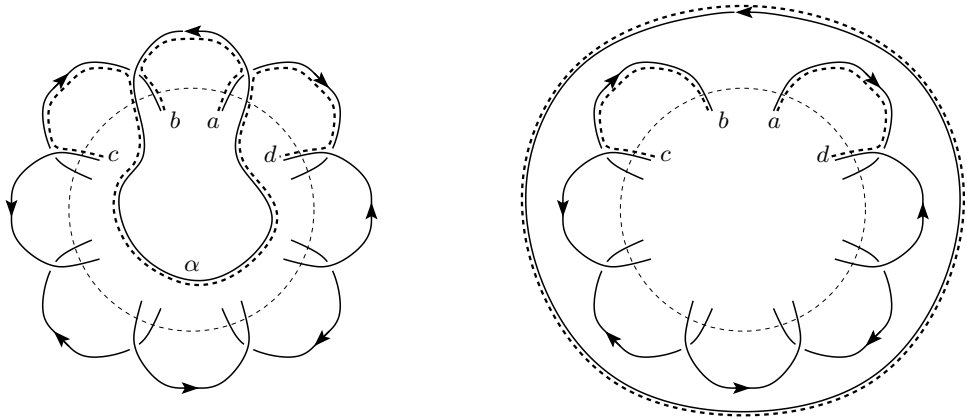} 
   \caption{Left: The diagram $D$ in case \ref{caseA}. The dashed arcs indicate pieces of components of $D'$. Right: The diagram $\widetilde D$, with several pieces of its Seifert circles.}
   \label{fig:befordit}
\end{figure}

We will show that $n(D)-s(D)\ge n(\widetilde D)-s(\widetilde D)+2$, i.e., that \eqref{eq:megvalami} for $D$ follows from applying Morton's inequality to $\widetilde D$. Let $N_\alpha$ be the number of crossings of $D$ along $\alpha$ (in particular, inside $C$). Let $D'$ be the diagram that results from smoothing all $N'$ crossings of $D$ away from $\alpha$ and inside $C$, as well as the $2k$ crossings outside of $C$. Let $s'$ be the number of components in $D'$. Then we have
\[n(D)-s(D)\ge2k+N_\alpha+N'-(s'+N_\alpha)=2k+N'-s',\]
because every time we smooth one of the remaining $N_\alpha$ crossings of $D'$, the number of components changes by $\pm1$.
On the other hand we have
\[n(\widetilde D)-s(\widetilde D)\le2k-2+N'-s',\]
from which our claim follows. To see why there are at least $s'$ Seifert circles in $\widetilde D$, notice that those components of $D'$ that do not pass through the points~$a$, $b$, $c$, or $d$ are found in $\widetilde D$ as well. There are $1$ or $2$ components of $D'$ that do pass through those points and $\widetilde D$ has at least $2$ components (the unknot we pulled out, and the one through, say, $a$) to account for them.

\item\label{caseB}
We assume that as we continue each outer over-arc of $D$ through the interior of $C$, we emerge at an outer under-arc. Let $\alpha_0$ be one of the arcs that we have just described. The endpoint of $\alpha_0$ is adjacent along $C$ to the startpoint $b_1$ of an outer over-arc. If we follow $D$ backward from $b_1$, then we hit either $C$ or $\alpha_1$ first. In the latter case, a noncyclic bigon is formed so that one of its corners is outside of $C$ but its other corner, as well as each arc of $D$ inside the bigon, is inside $C$. Therefore it is possible to change crossings inside $C$ so that a global noncyclic Reidemeister II-a move becomes possible for $D$. Hence in this case, we are done by Lemma \ref{lem:goodmoves} and Theorem \ref{thm:morton}.

In the case when the arc through $C$ that ends at $b_1$ is disjoint from $\alpha_0$, let us call the arc $\alpha_1$ and note that by our assumption in case \ref{caseB}, the startpoint $a_1$ of $\alpha_1$ has to be the endpoint of an outer under-arc. Now $a_1$ is adjacent along $C$ to the endpoint $a_2$ of an outer over-arc. Notice that $a_2$ and $\alpha_0$ are on opposite sides of $\alpha_1$. From here we iterate our argument: If the arc of $D$ that starts at $a_2$ hits $\alpha_1$ before $C$, then a noncyclic bigon and hence a good move can be found. Otherwise, follow the arc to its endpoint $b_2$ on $C$, which is necessarily on an outer under-arc, denote the arc $a_2b_2$ with $\alpha_2$, take the point $b_3$, adjacent along $C$ to $b_2$, where an outer over-arc starts, note that $\alpha_2$ separates $\alpha_1$ and $b_3$, and continue the iteration.

Our argument above produces a sequence $\alpha_0,\alpha_1,\ldots$ of parallel, disjoint chords in $C$ that are arranged monotonously. Since such a sequence cannot be infinite, it is guaranteed that after finitely many steps we find a global noncyclic Reidemeister~II-a move, which completes the proof of case \ref{caseB}.
\end{enumerate}

In the rest of the proof, we are going to verify \eqref{eq:megvalami} by a secondary induction on $k$ (keeping $N$ fixed). When $k=1$, the diagram $D$ falls under one of the cases \ref{caseA} and \ref{caseB} above.

Let us now assume that \eqref{eq:megvalami} holds whenever the number of outer over-arcs is less than $k$ and let $D$ be a link diagram with an alternating contour and $k$ outer over-arcs. Having established cases \ref{caseA} and \ref{caseB}, we may assume that there exists an outer over-arc in $D$ so that the arc $\alpha$ across $C$ that starts at its endpoint will end at a different outer over-arc. 

\begin{figure}[htbp]
\labellist
\small
\pinlabel $\alpha$ at 200 230
\endlabellist
  \includegraphics[width=\linewidth]{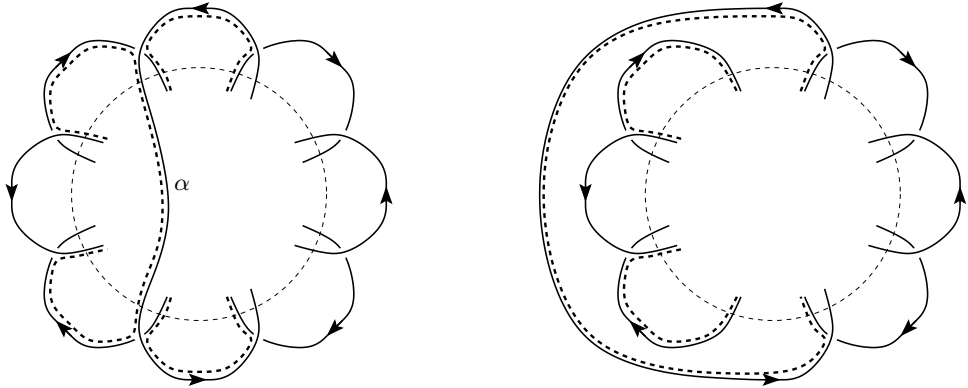}
  \caption{Left: Diagram $D$ with an alternating contour. Dashed arcs indicate components of $D'$. Right: the diagram $\widetilde D$.}
  \label{fig:outer_arcs_over}
\end{figure}

Given a diagram $D$ as in the left panel of Figure~\ref{fig:outer_arcs_over}, let $s'$ be the number of components in the diagram $D'$ that results from smoothing all crossings in $D$ except for those along $\alpha$.
Let $N_\alpha$ be the number of crossings on $\alpha$ and let $N'$ be the number of crossings of $D$ inside $C$ that are not on $\alpha$.
Then there is a total of $2k+N_\alpha+N'$ crossings in $D$ and the number of Seifert circles is at most $s'+N_\alpha$ by the same reason as before.
Hence we have
\begin{equation*}
  2k+N'-s'
  \le
  n(D)-s(D).
\end{equation*}


After changing crossings if necessary, let us pull $\alpha$ out as in the right panel of Figure~\ref{fig:outer_arcs_over} and call the resulting diagram $\widetilde{D}$.
Then the number of crossings in $\widetilde{D}$ is $2k-2+N'$ and the number of Seifert circles is at least $s'-2$. (It is possible that the arcs of $D'$ indicated in the left panel of Figure \ref{fig:outer_arcs_over} belong to three different components, whereas the arcs of $\widetilde D$ indicated on the right belong to just one component. But it is easy to see that nothing worse than that can happen.)
Therefore we have
\begin{equation*}
  n(\widetilde{D})-s(\widetilde{D})
  \le
  2k+N'-s',
\end{equation*}
meaning that the isotopy from $D$ to $\widetilde D$ was at least a fair move. But since $\widetilde D$ is also a diagram with an alternating contour but with less than $k$ outer over-arcs, we are done by Lemma \ref{lem:fair} and the inductive hypothesis.

This completes the secondary induction (on $k$), thus we have proved \eqref{eq:megvalami} and hence the inductive step in our main induction (on $N$) is now also established.
\end{proof}


\begin{thebibliography}{33}




\bibitem{jonevevan}D. Chebikin and P. Pylyavskyy, \textit{A family of bijections between $G$-parking functions and spanning trees}, J. Combin.\ Theory Ser.\ A 110 (2005), no.\ 1, 31--41.

\bibitem{cromwell}
P. R. Cromwell, \textit{Homogeneous links}, J. London Math.\ Soc.\ 39 (1989), no.\ 3, 535--552.



\bibitem{gv}R. Gopakumar and C. Vafa, \textit{On the gauge theory/geometry correspondence}, Adv.\ Theor.\ Math.\ Phys.\ 3 (1999), no.\ 5, 1415--1443.

\bibitem{jaeger}F.\ Jaeger, \textit{Tutte polynomials and link polynomials}, Proc.\ Amer.\ Math.\ Soc.\ 103 (1988), no.\ 2, 647--654.


\bibitem{jones}V.\ Jones, \textit{Hecke algebra representations of braid groups and link polynomials}, Ann.\ Math.\ 126 (1987), 335--388.

\bibitem{jkr}A. Juh\'asz, T. K\'alm\'an, and J. Rasmussen, \textit{Sutured Floer homology and hypergraphs}, to appear in Math.\ Res.\ Lett.

\bibitem{hypertutte}T. K\'alm\'an, \textit{A version of Tutte's polynomial for hypergraphs}, arXiv:1103.1057, to appear in Adv.\ Math.





\bibitem{merino}C. Merino L\'opez, \textit{Chip firing and the Tutte polynomial}, Ann.\ Comb.\ 1 (1997), no.\ 3, 253--259.

\bibitem{morton}H. R. Morton, \textit{Seifert circles and knot polynomials},
Math.\ Proc.\ Camb.\ Phil.\ Soc.\ 99 (1986), 107--109.

\bibitem{mp}K.\ Murasugi and J.\ Przytycki, \textit{The skein polynomial of a planar star product of two links}, Math.\ Proc.\ Camb.\ Phil.\ Soc.\ 106 (1989), 273--276.

\bibitem{ov}H. Ooguri and C. Vafa, \textit{Knot invariants and topological strings}, Nucl.\ Phys.\ B 577 (2000), no.\ 3, 419--438.

\bibitem{post}A.\ Postnikov, \textit{Permutohedra, Associahedra, and Beyond}, Int.\ Math.\ Res.\ Not.\ 2009, no.\ 6, 1026--1106.

\bibitem{ps}A.\ Postnikov and B.\ Shapiro, \textit{Trees, parking functions, syzygies, and deformations of monomial ideals}, Trans.\ Amer.\ Math.\ Soc.\ 356 (2004), no.\ 8, 3109--3142.





\bibitem{swartz}E. Swartz, \textit{From polytopes to enumeration}, {\tt http://www.math.cornell.edu/\raisebox{-2pt}{\~{}}ebs/math455.pdf}.






\end{thebibliography}
\end{document}